\documentclass{article}
\usepackage{amssymb}
\usepackage{amsmath}
\usepackage{amsfonts}
\usepackage{geometry}
\usepackage{appendix}
\usepackage{hyperref}

\newtheorem{theorem}{Theorem}

\newtheorem{condition}[theorem]{Condition}

\newtheorem{definition}[theorem]{Definition}
\newtheorem{example}[theorem]{Example}

\newtheorem{lemma}[theorem]{Lemma}

\newtheorem{proposition}[theorem]{Proposition}
\newtheorem{remark}[theorem]{Remark}

\newenvironment{proof}[1][Proof]{\noindent\textbf{#1.} }{\ \rule{0.5em}{0.5em}}
\input{tcilatex}
\begin{document}

\title{On the splitting-up method for rough (partial) differential equations}
\author{Peter Friz \\
TU Berlin and WIAS \and Harald Oberhauser \\
TU Berlin}
\maketitle

\begin{abstract}
This article introduces the splitting method to systems responding to rough
paths as external stimuli. The focus is on nonlinear partial differential
equations with rough noise but we also cover rough differential equations.
Applications to stochastic partial differential equations arising in control
theory and nonlinear filtering are given.
\end{abstract}

\section{\protect\bigskip Introduction}

This article introduces the splitting-up method for systems responding to
rough paths as external stimuli. We deal with (nonlinear) rough partial
differential equations, RPDEs, formally written as%
\begin{equation}
du=F\left( t,x,u,Du,D^{2}u\right) dt+\Lambda \left( t,x,u,Du\right) d\mathbf{%
z}\text{ on }\left( 0,T\right] \times \mathbb{R}^{e},\text{ }u\left(
0,x\right) =u_{0}\left( x\right)  \label{EqSPDEIntro}
\end{equation}%
but we also cover rough differential equations, RDEs,\bigskip\ of the form%
\begin{equation*}
dy_{t}=V\left( y_{t}\right) dt+W\left( y_{t}\right) d\mathbf{z}_{t}.
\end{equation*}%
In both examples, $\mathbf{z}$ is an external stimuli given as a rough path, 
$F$ is a nonlinear (possibly degenerate) elliptic operator, $\Lambda $ is a
collection of affine linear operators, i.e.%
\begin{equation*}
\Lambda _{k}\left( t,x,r,p\right) =\left( p\cdot \sigma _{k}\left(
t,x\right) \right) +r\,\nu _{k}\left( t,x\right) +g_{k}\left( t,x\right) 
\text{.}
\end{equation*}%
and $\sigma ,\nu ,g$ resp.\ $V,W=\left( W_{i}\right) $ are (collections) of
vector fields on $\left[ 0,T\right] \times \mathbb{R}^{e}$ resp.\ $\mathbb{R}%
^{e}$. Consequences of our results are splitting results for (nonlinear)
stochastic partial differential equations, SPDEs, that is, when $\mathbf{z}$
is taken to be a rough path lift of a stochastic process (e.g.\ fractional
BM with Hurstparameter $>\frac{1}{4}$); for example, linear SPDEs with
Brownian noise in Stratonovich form (e.g. the\ Zakai equation from nonlinear
filtering),%
\begin{equation}
du=L\left( t,x,u,Du,D^{2}u\right) dt+\sum_{k=1}^{d}\Lambda _{k}\left(
t,x,u,Du\right) \circ dB^{k},\text{ }u\left( 0,.\right) =u_{0}\left(
.\right) \text{.}  \label{EqLinearSPDE}
\end{equation}%
are covered. $L$ is here a linear (degenerate elliptic) operator, 
\begin{equation*}
L\left( t,x,r,p,X\right) =\text{Trace}\left[ A\left( t,x\right) \cdot X%
\right] +b\left( t,x\right) \cdot p+c\left( t,x,r\right) ,
\end{equation*}%
and $B$ a standard $d$-dimensional Brownian motion. The splitting-up method
(which runs under many names: dimensional splitting, operator splitting,
Lie-Trotter-Kato formula, Baker-Campbell-Hausdorff formula, Chernoff
formula, leapfrog method, predictor-corrector method, etc.) is one of the
most prominent methods for calculating solutions of (stochastic-, ordinary-,
partial-) differential equations numerically; for a survey we recommend \cite%
{McLachlanQuispel02:SplittingMethods}. For S(P)DEs a splitting-up method was
introduced in \cite{BensoussanGlowinski89:ApprOfZakaiSplitting} for the
Zakai equation in filtering and has received much attention since. We
explicitly mention \cite{KrylovGyongy03:OnSplittingUp} which extends the
previous results to general linear SPDEs of the form $\left( \ref%
{EqLinearSPDE}\right) $.

Let us informally describe the general idea of splitting using $\left( \ref%
{EqLinearSPDE}\right) $: For $n\in \mathbb{N}$ consider the partition $%
D^{n}=\left\{ t_{i}^{n}=in^{-1}T,i=0,\ldots ,n\right\} $ of the interval $%
\left[ 0,T\right] $ and define the approximation $u_{n}$ recursively 
\begin{equation*}
u_{n}\left( t_{i+1}^{n},.\right) :=\left[ \mathbf{Q}_{t_{i}^{n}t_{i+1}^{n}}%
\circ \mathbf{P}_{t_{i}^{n}t_{i+1}^{n}}\right] \left( u_{n}\left(
t_{i}^{n},.\right) \right)
\end{equation*}%
with $\left\{ \mathbf{P}_{st},0\leq s\leq t\leq T\right\} $ and $\left\{ 
\mathbf{Q}_{st},0\leq s\leq t\leq T\right\} $ the solution operators of 
\begin{eqnarray}
dv &=&L\left( t,x,v,Dv,D^{2}v\right) dt,\text{ }v\left( s,x\right) =v\left(
x\right) \text{ and}  \label{Eq_pde} \\
dw &=&\sum_{k=1}^{d}\Lambda _{k}\left( t,x,w,Dw,D^{2}w\right) \circ
dB_{t}^{k},\text{ }w\left( s,x\right) =v\left( x\right) \text{.}
\label{Eq_Stoch}
\end{eqnarray}%
That is, on each interval $\left[ t_{i}^{n},t_{i+1}^{n}\right] $ one solves
first equation $\left( \ref{Eq_Stoch}\right) $ on $\left[
t_{i}^{n},t_{i+1}^{n}\right] $ with initial value $u_{n}\left(
t_{i}^{n},.\right) $ and then one uses its solution as initial value for the
PDE $\left( \ref{Eq_pde}\right) $ (so-called \textquotedblleft predictor"
and \textquotedblleft corrector" steps in \cite{FlorchingerLeGland91}).
Under appropiate conditions, one can show that $u_{n}$ converges to $u$ and
also derive rates of convergence, \cite{KrylovGyongy03:OnSplittingUp}.

All the above mentioned authors use (to the best of our knowledge) either
semigroup theory or stochastic calculus to prove splitting results but
neither are available for (\ref{EqSPDEIntro}) due to the nonlinear operator $%
F$ and the nonsemimartingale noise $\mathbf{z}$. The point of view of this
article is different; loosely speaking: \emph{splitting-up results follow
from stability in a rough path sense}. We combine the method of Krylov and Gy%
\"{o}ngy, \cite{KrylovGyongy03:OnSplittingUp}, of stretching out the
time-scale with certain stability results of RPDEs. Applications to SPDEs
then follow and our results are, to the best of our knowledge, new for
nonlinear PDEs with noise of above form (see \ also \cite%
{LionsSoug98:ViscSPDEfirstOne,
LionsSoug98:ViscSPDEsecondOne,LionsSoug98:ViscSPDEsecondOne,LionsSoug00:ViscSPDESeminlinearNoise,LionsSoug00:ViscSPDEUniqueness}%
). Due to the generality of equation $\left( \ref{EqSPDEIntro}\right) $ we
do not give rates of convergence but hope to return to this question in the
future.

\subsection{Some ideas from rough path and viscosity theory}

Let us recall some basic ideas of (second order) viscosity theory \cite%
{MR1118699UserGuide, MR2179357FS} and rough path theory \cite{lyons-qian-02,
MR2314753}.\ As for viscosity theory, consider a real-valued function $%
u=u\left( t,x\right) $ with $t\in \left[ 0,T\right] ,x\in \mathbb{R}^{e}$
and assume $u\in C^{2}$ is a classical subsolution,%
\begin{equation*}
\partial _{t}u+F\left( t,x,u,Du,D^{2}u\right) \leq 0,
\end{equation*}%
where $F$ is a (continuous) function, \textit{degenerate elliptic} in the
sense that $F\left( t,x,r,p,A+B\right) \leq F\left( t,x,r,p,A\right) $
whenever $B\geq 0$ in the sense of symmetric \ matrices. The idea is to
consider a (smooth) test function $\varphi $ and look at a local maxima $%
\left( \hat{t},\hat{x}\right) $ of $u-\varphi $. Basic calculus implies that 
$Du\left( \hat{t},\hat{x}\right) =D\varphi \left( \hat{t},\hat{x}\right)
,\,D^{2}u\left( \hat{t},\hat{x}\right) \leq D\varphi \left( \hat{t},\hat{x}%
\right) $ and, from degenerate ellipticity,%
\begin{equation}
\partial _{t}\varphi +F\left( \hat{t},\hat{x},u,D\varphi ,D^{2}\varphi
\right) \leq 0.  \label{Gxbar}
\end{equation}%
This suggests to define a \textit{viscosity supersolution} (at the point $%
\left( \hat{x},\hat{t}\right) $) to $\partial _{t}+F=0$ as a continuous
function $u$ with the property that (\ref{Gxbar}) holds for any test
function. Similarly, \textit{viscosity subsolutions} are defined by
reversing inequality in (\ref{Gxbar}); \textit{viscosity solutions} are both
super- and subsolutions. A different point of view is to note that $u\left(
t,x\right) \leq u\left( \hat{t},\hat{x}\right) -\varphi \left( \hat{t},\hat{x%
}\right) +\varphi \left( t,x\right) $ for $\left( t,x\right) $ near $\left( 
\hat{t},\hat{x}\right) $. A simple Taylor expansion then implies%
\begin{equation}
u\left( t,x\right) \leq u\left( \hat{t},\hat{x}\right) +a\left( t-\hat{t}%
\right) +p\cdot \left( x-\hat{x}\right) +\frac{1}{2}\left( x-\hat{x}\right)
^{T}\cdot X\cdot \left( x-\hat{x}\right) +o\left( \left\vert \hat{x}%
-x\right\vert ^{2}+\left\vert \hat{t}-t\right\vert \right)  \label{EqTaylor}
\end{equation}%
as $\left\vert \hat{x}-x\right\vert ^{2}+\left\vert \hat{t}-t\right\vert
\rightarrow 0$ with $a=\partial _{t}\varphi \left( \hat{t},\hat{x}\right) ,$ 
$p=D\varphi \left( \hat{t},\hat{x}\right) ,$ $X=D^{2}\varphi \left( \hat{t},%
\hat{x}\right) $. Moreover, if (\ref{EqTaylor}) holds for some $\left(
a,p,X\right) $ and $u$ is differentiable, then $a=\partial _{t}u\left( \hat{t%
},\hat{x}\right) ,$ $p=Du\left( \hat{t},\hat{x}\right) ,$ $X\leq
D^{2}u\left( \hat{t},\hat{x}\right) $, hence by degenerate ellipticity%
\begin{equation*}
\partial _{t}\varphi +F\left( \hat{t},\hat{x},u,p,X\right) \leq 0\text{.}
\end{equation*}%
Pushing this idea further leads to a definition of viscosity solutions based
on a generalized notion of \textquotedblleft $\left( \partial
_{t}u,Du,D^{2}u\right) $" for nondifferentiable $u$, the so-called parabolic
semijets, and it is a simple exercise to show that both definitions are
equivalent. The resulting theory (existence, uniqueness, stability, ...) is
without doubt one of the most important recent developments in the field of
partial differential equations. As a typical result\footnote{$\mathrm{BUC}%
\left( \dots \right) $ denotes the space of bounded, uniformly continuous
functions.}, the initial value problem $\left( \partial _{t}+F\right)
u=0,\,u\left( 0,\cdot \right) =u_{0}\in \mathrm{BUC}\left( \mathbb{R}%
^{e}\right) $ has a unique solution in $\mathrm{BUC}\left( [0,T]\times 
\mathbb{R}^{e}\right) $ provided $F=F(t,x,u,Du,D^{2}u)$ is continuous,
degenerate elliptic, proper (i.e. increasing in the $u$ variable) and
satisfies a (well-known) technical condition\footnote{%
(3.14) of the User's Guide \cite{MR1118699UserGuide}.}. In fact, uniqueness
follows from a stronger property known as \textit{comparison:} assume $u$
(resp.\ $v$) is a supersolution (resp.\ subsolution) and $u_{0}\geq v_{0}$;
then $u\geq v$ on $[0,T]\times \mathbb{R}^{e}$. A key feature of viscosity
theory is what workers in the field simply call \textit{stability properties}%
. For instance, it is relatively straight-forward to study $\left( \partial
_{t}+F\right) u=0$ via a sequence of approximate problems, say $\left(
\partial _{t}+F^{n}\right) u^{n}=0$, provided $F^{n}\rightarrow F$ locally
uniformly and some apriori information on the $u^{n}$ (e.g.\ locally uniform
convergence, or locally uniform boundedness\footnote{%
What we have in mind here is the \textit{Barles--Perthame method of
semi-relaxed limits} \cite{MR2179357FS}.}. Note the stark contrast to the
classical theory where one has to control the actual derivatives of $u^{n}$.

The notion of stability is also central to rough path theory. Let $y_{0}\in 
\mathbb{R}^{e},V,W=\left( W_{i}\right) _{i=1,\ldots ,d}$ be (collections of)
vector fields on $\mathbb{R}^{e}$. Using rough path theory\footnote{%
Young theory is actually good enough for $p<2;$ see \cite{lyons-94}$.$}, one
can speak of solutions to 
\begin{equation}
dy_{t}=V\left( y_{t}\right) d\xi +W\left( y_{t}\right) d\mathbf{z}_{t},\text{
}y\left( 0\right) =y_{0}\in \mathbb{R}^{e}  \label{Eq_ODE_drift}
\end{equation}%
if the weak geometric rough path $\mathbf{z\in }C^{p\text{-var}}\left( \left[
0,T\right] ,G^{\left[ p\right] }\left( \mathbb{R}^{d}\right) \right) $ and $%
\xi \in C^{q\text{-var}}\left( \left[ 0,T\right] ,\mathbb{R}\right) $ have
"complementary Young regularity" $\frac{1}{p}+\frac{1}{q}>1$: then $\left(
\xi ,\mathbf{z}\right) $ can be seen as "time-space" rough path, i.e.\ an
element of $C^{\max \left( p,q\right) }\left( \left[ 0,T\right] ,G^{\left[
\max (p,q)\right] }\left( \mathbb{R\oplus R}^{d}\right) \right) $ since all
necessary cross iterated integrals between $\mathbf{z}$ and $\xi $ are
well-defined (using Young integration; \cite{frizvictoirbook},\cite%
{MR2228694}). Now any sequence $\left( \xi ^{n},z^{n}\right) _{n}\subset C^{1%
\text{-var}}\left( \left[ 0,T\right] ,\mathbb{R}\right) \times C^{1\text{-var%
}}\left( \left[ 0,T\right] ,\mathbb{R}^{d}\right) $ leads to a sequence of
ODE solutions $\left( y^{n}\right) _{n}$ of%
\begin{equation*}
dy_{t}^{n}=V\left( y_{t}^{n}\right) d\xi _{t}^{n}+\sum_{i=1}^{d}W_{i}\left(
y_{t}^{n}\right) dz_{t}^{n,i}
\end{equation*}%
(again assuming the vector fields are regular enough) and we call any limit
point in uniform topology of $\left( y^{n}\right) _{n}$ a solution of the RDE%
\begin{equation*}
dy_{t}=V\left( y_{t}\right) d\xi _{t}+W\left( y_{t}\right) d\mathbf{z}_{t},
\end{equation*}%
if $\left( \xi ^{n},z^{n}\right) $ converges to $\left( \xi ,\mathbf{z}%
\right) $ in the sense%
\begin{eqnarray}
\sup_{n}\left\vert \left\vert S_{\left[ p\right] }\left( z^{n}\right)
\right\vert \right\vert _{p-\text{var}}+\sup_{n}\left\vert \left\vert S_{%
\left[ q\right] }\left( \xi ^{n}\right) \right\vert \right\vert _{q-\text{var%
}} &<&\infty  \label{EqCondYoungPair} \\
d_{0}\left( S_{\left[ p\right] }\left( z^{n}\right) ,\mathbf{z}\right)
+\left\vert \xi ^{n}-\xi \right\vert _{\infty ;\left[ 0,T\right] }
&\rightarrow &0\text{ as }n\rightarrow \infty .  \notag
\end{eqnarray}%
Here, $S_{\left[ p\right] }$ denotes the canonical lift to a geometric step-$%
\left[ p\right] $ rough path (i.e.\ given by Riemann-Stieltjes integration),%
\begin{equation}
S_{\left[ p\right] }\left( z^{n}\right)
_{s,t}=1+\int_{s}^{t}dz_{u_{1}}^{n}+\cdots +\int_{s\leq u_{1}\leq \cdots
\leq u_{\left[ p\right] }\leq t}dz_{u_{1}}^{n}\otimes \cdots \otimes dz_{u_{%
\left[ p\right] }}^{n}.  \label{EqLift}
\end{equation}%
If there exists a unique solution we denote it by $\pi _{V,\left( W\right)
}\left( 0,x;\left( \xi ,\mathbf{z}\right) \right) $ to emphasize dependence
on the initial condition $y_{0}$, the rough path $\left( \xi ,\mathbf{z}%
\right) $, the vector field $V$ and the collection of vector fields $%
W=\left( W_{i}\right) _{i=1}^{d}$. A variation of Lyons' limit theorem, \cite%
{frizvictoirbook}, gives sufficient\footnote{%
...and essentially sharp (see \cite{Davie06:DiscreteApproxRoughPaths}). We
remark also that one can slightly improve on regularity assumption of the
vector fields by reproving the Universal Limit theorem for RDEs with drift
instead of using the "space-time" rough path $\left( \mathbf{z},\xi \right) $%
} conditions for existence and uniqueness of such RDE solutions.

\section{Two examples}

In this section we sketch our approach on two examples, ODEs and RDEs
(resp.\ SDEs when take the rough path lift of a stochastic process).
Therefore$\ $let us introduce some notation: for fixed $\Delta >0$ and $t\in %
\left[ 0,T\right] $ set\footnote{$\left\lfloor .\right\rfloor $ denotes the
lower floor function.} $t_{\Delta }=\left\lfloor t/\Delta \right\rfloor
\Delta $ and $t^{\Delta }=\left\lfloor t/\Delta \right\rfloor \Delta +\Delta 
$ (i.e.\ $\left[ t_{\Delta },t^{\Delta }\right) $ is the interval in the
partition of $\left[ 0,T\right] $ of constant mesh size $\Delta $ that
contains $t$). Motivated by \cite{KrylovGyongy03:OnSplittingUp} define two
time changes

\begin{equation*}
a\left( \Delta ,t\right) \bigskip =\left\{ 
\begin{array}{cc}
t_{\Delta }+2\left( t-t_{\Delta }\right) , & t_{\Delta }\leq t\leq t_{\Delta
}+\Delta /2 \\ 
t^{\Delta }, & t_{\Delta }+\Delta /2<t\leq t^{\Delta }%
\end{array}%
\right. ,\text{ }b\left( \Delta ,t\right) \bigskip =a\left( \Delta ,t+\frac{%
\Delta }{2}\right) \text{.}
\end{equation*}%
That is, $a\left( \Delta ,.\right) $ runs on the first half of each
intervall $\left[ t_{\Delta },t^{\Delta }\right) $ with double speed from $%
t_{\Delta }$ to $t^{\Delta }$ and stays still in the second half, wheras $%
b\left( \Delta ,.\right) $ does this in opposite order; also the paths $%
\left( a\left( \Delta ,\cdot \right) ,b\left( \Delta ,\cdot \right) \right) $
converge in $\left( 1+\varepsilon \right) $-variation against the path $%
id_{2}:t\longmapsto \left( t,t\right) $ for every $\varepsilon >0$.

Further, for given $n\geq 1$ denote by $D^{n}$ the partition $\left\{ \frac{k%
}{n}T,k=0,\ldots ,n\right\} $ of $\left[ 0,T\right] $.

The following two examples were shown to us by Terry Lyons at the IRTG SMCP
Summer School 2009 in Chorin:

\subsection{Splitting ODEs}

Let $V,W\in Lip^{1}\left( \mathbb{R}^{e},\mathbb{R}^{e}\right) $. We are
interested in splitting of the ODE 
\begin{equation}
dy_{t}=V\left( y_{t}\right) dt+W\left( y_{t}\right) dt,\text{ }y\left(
0\right) =y_{0}\in \mathbb{R}^{e}.  \label{EqODE}
\end{equation}%
As in the introduction, denote the solution of $\left( \ref{EqODE}\right) $
as $\pi _{V,W}\left( 0,y_{0};id_{2}\right) $. Classic Lie splitting
corresponds to the approximation of the path $id_{2}$ by the sequence of
paths $t\longmapsto \left( a\left( n^{-1},t\right) ,b\left( n^{-1},t\right)
\right) $. Therefore let $y^{n}$ be the ODE solution of%
\begin{equation*}
dy_{t}^{n}=V\left( y_{t}\right) da\left( n^{-1},t\right) +W\left(
y_{t}\right) db\left( n^{-1},t\right) ,\text{ }y\left( 0\right) =y_{0}\in 
\mathbb{R}^{e}\text{,}
\end{equation*}%
i.e.\ $y^{n}=$ $\pi _{V,W}\left( 0,y_{0};\left( a\left( n^{-1},\cdot \right)
,b\left( n^{-1},\cdot \right) \right) \right) $. For brevity, define the
solution operator $\left\{ \mathbf{P}_{s,t}^{n;V},0\leq s\leq t\leq
T\right\} $ and $\left\{ \mathbf{P}_{u}^{V},0\leq u\leq T\right\} $ mapping
points in $\mathbb{R}^{e}$ to $\mathbb{R}^{e}$ as%
\begin{equation*}
\mathbf{P}_{s,t}^{n,V}\left( x\right) :=\pi _{V}\left( s,x;a\left(
n^{-1},\cdot \right) \right) _{t}\text{ and }\mathbf{P}_{t-s}^{V}\left(
x\right) :=\pi _{V}\left( 0,x;id_{1}\right) _{t-s}=\pi _{V}\left(
s,x;id_{1}\right) _{t}
\end{equation*}%
(here $id_{1}:t\longmapsto t$ and $\mathbf{P}^{V}$ is a one parameter
semigroup due to homogenity of $id_{1}$); similarly define $\mathbf{Q}^{n,W}$
and $\mathbf{Q}^{W}$. Firstly, note that by defintion of $a$ and $b,$ 
\begin{equation}
y^{n}\left( t+\frac{1}{n}\right) =\left[ \mathbf{Q}_{t,t+1/n}^{n;W}\circ 
\mathbf{P}_{t,t+1/n}^{n;V}\right] \left( y^{n}\left( t\right) \right)
\label{Eq_Split}
\end{equation}%
for $t$ being a multiple of $T/n.$ Secondly, it is intuitively clear, and an
easy exercise to show, that $\mathbf{P}_{s,t}^{n,V}\equiv \mathbf{P}%
_{t-s}^{V}$ resp.\ $\mathbf{Q}_{s,t}^{n,W}\equiv \mathbf{Q}_{t-s}^{W}$ for $%
s,t\ $points on the dissection $D^{n}$. Since $\left( a\left( n^{-1},\cdot
\right) ,b\left( n^{-1},\cdot \right) \right) $ converges to the path $%
id_{2}:t\longmapsto \left( t,t\right) $ in the sense of $\left( \ref%
{EqCondYoungPair}\right) $ (that is, $\xi ^{n}\left( t\right) =a\left(
n^{-1},t\right) $ and $z^{n}\left( t\right) =b\left( n^{-1},t\right) $, $%
q=p=1$), we know that 
\begin{equation*}
\pi _{V,W}\left( 0,y_{0};\left( a\left( n^{-1},\cdot \right) ,b\left(
n^{-1},\cdot \right) \right) \right) =y^{n}\rightarrow y=\pi _{V,W}\left(
0,y_{0};id_{2}\right) \text{ as }n\rightarrow \infty
\end{equation*}%
in $\left\vert .\right\vert _{\infty ;\left[ 0,T\right] }$ norm. Using the
identity $\left( \ref{Eq_Split}\right) $ one recovers the \textquotedblright
classic Lie-splitting" 
\begin{equation*}
\left[ \mathbf{Q}_{1/n}^{W}\circ \mathbf{P}_{1/n}^{V}\right] ^{\left\lfloor
t/n\right\rfloor }\left( y_{0}\right) \rightarrow y_{t}\text{ as }%
n\rightarrow \infty \text{ for every }t\in \left[ 0,T\right]
\end{equation*}%
where $y$ is the ODE\ solution of $\left( \ref{EqODE}\right) $. Moreover,
the convergence holds in $\left\vert .\right\vert _{\infty ;\left[ 0,T\right]
}$ norm and by interpolation even in stronger $\left( 1+\varepsilon \right) $%
-variation norm for every $\varepsilon >0$.

Note that no rough path theory is needed and everything follows from
continuouity in the sense of $\left( \ref{EqCondYoungPair}\right) $ (with $%
q=1$) which can be established by elementary computations; see \cite[Chapter
3]{frizvictoirbook}. Let us remark that using Young integration theory one
can push this method to driving signals of finite $p-$variation for $p\lneqq
2$. Paths of Brownian regularity or worse are outside the scope of Young
theory but one can use rough path results.

\subsection{Splitting RDEs}

Motivated by the above example we can ask for splitting for an RDE with
drift of the form%
\begin{equation}
dy_{t}=V\left( y_{t}\right) d\xi +W\left( y_{t}\right) d\mathbf{z}_{t},\text{
}y\left( 0\right) =y_{0}\in \mathbb{R}^{e},  \label{Eq_drift}
\end{equation}%
where $\mathbf{z\in }C^{p\text{-var}}\left( \left[ 0,T\right] ,G^{\left[ p%
\right] }\left( \mathbb{R}^{d}\right) \right) $ and $\xi \mathbf{\in }C^{1%
\text{-var}}\left( \left[ 0,T\right] ,G^{\left[ p\right] }\left( \mathbb{R}%
^{d}\right) \right) $; the path $\left( \xi ,\mathbf{z}\right) $ has
trivially "complementary Young regularity" and if $V\in Lip^{\gamma }\left( 
\mathbb{R}^{e},\mathbb{R}^{e}\right) $ and $W=\left( W_{i}\right) \subset
Lip^{\tilde{\gamma}}\left( \mathbb{R}^{e},\mathbb{R}^{e}\right) $ for $%
\gamma >1,\tilde{\gamma}>p$ then there exists a unique solution $\pi
_{V,\left( W\right) }\left( 0,y_{0};\left( \xi ,\mathbf{z}\right) \right) $
to $\left( \ref{Eq_drift}\right) $. For later use we show

\begin{lemma}
\label{Lemma_Joint_RP_convergence}Let $\mathbf{z\in }C^{p\text{-var}}\left( %
\left[ 0,T\right] ,G^{\left[ p\right] }\left( \mathbb{R}^{d}\right) \right)
, $ $\xi \in C^{q\text{-var}}\left( \left[ 0,T\right] ,\mathbb{R}\right) $.
If we define $\xi ^{\Delta }\left( t\right) =\xi \left( a\left( \Delta
,t\right) \right) \in C^{q\text{-var}}\left( \left[ 0,T\right] ,\mathbb{R}%
\right) $, $\mathbf{z}^{\Delta }\left( t\right) =\mathbf{z}\left( b\left(
\Delta ,t\right) \right) \in C^{p\text{-var}}\left( \left[ 0,T\right] ,G^{%
\left[ p\right] }\left( \mathbb{R}^{d}\right) \right) $ then%
\begin{eqnarray*}
\sup_{\Delta >0}\left\vert \left\vert \mathbf{z}^{\Delta }\right\vert
\right\vert _{p\text{-var}}+\sup_{\Delta >0}\left\vert \xi ^{\Delta
}\right\vert _{q\text{-var}} &<&\infty \\
d_{0}\left( \mathbf{z}^{\Delta },\mathbf{z}\right) +\left\vert \xi ^{\Delta
}-\xi \right\vert _{\infty ;\left[ 0,T\right] } &\rightarrow &0\text{ as }%
\Delta \rightarrow 0.
\end{eqnarray*}
\end{lemma}

\begin{proof}
First note that the variation norm is invariant under reparametrisation
which implies the first statement. For the second statement it is sufficient
to show pointwise converge (by interpalation pointwise convergence in
combination with uniform variation bounds implies convergence in supremum
norm). However, $\mathbf{z}$ and $\xi $ are, by assumption, both continuous
paths which gives pointwise convergence.
\end{proof}

Define 
\begin{equation*}
\xi ^{n}\left( t\right) :=\xi \left( a\left( n^{-1},t\right) \right) \text{
and }\mathbf{z}_{t}^{n}:=\mathbf{z}\left( b\left( n^{-1},t\right) \right) 
\text{.}
\end{equation*}%
Similar to the ODE example, define the solution operator $\left\{ \mathbf{P}%
_{s,t}^{n;V},0\leq s\leq t\leq T\right\} $ mapping points in $\mathbb{R}^{e}$
to $\mathbb{R}^{e}$ as $\mathbf{P}_{s,t}^{n,V}\left( x\right) :=\pi
_{V}\left( s,x;\xi ^{n}\right) _{t}$ and the operators $\mathbf{P}^{V}$,$%
\mathbf{Q}^{n,W}$ and $\mathbf{Q}^{W}$. It remains to show that 
\begin{equation*}
\mathbf{P}_{s,t}^{n;V}\equiv \mathbf{P}_{s,t}^{V}\text{ and }\mathbf{Q}%
_{s,t}^{n;W}\equiv \mathbf{Q}_{s,t}^{W}\text{ for }s,t\in D^{n}
\end{equation*}%
(in contrast with the ODE example, $\mathbf{P}^{V}$ and $\mathbf{Q}^{W}$ are
now two parameter semigroups due to the time-inhomogenity of $\xi $ and $%
\mathbf{z}$). Since $G^{\left[ p\right] }\left( \mathbb{R}^{d}\right) $ is a
geodesic space, there exists a sequence of paths (concatenations of
geodesics on the sequence of dissections $D^{m}$), $\left( z^{m}\right)
_{m}\subset C^{1-\text{var}}\left( \left[ 0,T\right] ,\mathbb{R}^{d}\right) $
with $S_{\left[ p\right] }\left( z_{s,t}^{m}\right) =\mathbf{z}_{s,t}$ for $%
s,t\in D^{m}$, such that%
\begin{eqnarray*}
\sup_{m}\left\vert \left\vert S_{\left[ p\right] }\left( z^{m}\right)
\right\vert \right\vert _{p-\text{var}} &<&\infty , \\
d_{0}\left( S_{\left[ p\right] }\left( z^{m}\right) ,\mathbf{z}\right)
&\rightarrow &0\text{ as }m\rightarrow \infty ,
\end{eqnarray*}%
with $S_{\left[ p\right] }$ as in $\left( \ref{EqLift}\right) $. Now define $%
\mathbf{z}^{n,m}\left( t\right) :=\mathbf{z}^{m}\left( a\left(
n^{-1},t\right) \right) $ and note that $\mathbf{z}^{m}$ and $\mathbf{z}%
^{n,m}$ have bounded 1-variation ($\mathbf{z}^{m}$ by construction, $\mathbf{%
z}^{n,m}$ because the variation norm is invariant under reparametrisation).
Hence, for $m,n$ fix we deal with an ODE as in the example above and
therefore%
\begin{equation*}
\pi _{W}\left( s,x;\mathbf{z}^{n,m}\right) _{s,t}=\pi _{W}\left( s,x;\mathbf{%
z}^{m}\right) _{s,t}\text{ }
\end{equation*}%
for $s,t\in D^{n}$. Keeping $n$ fixed and letting $m\rightarrow \infty ,$
the LHS converges to $\pi _{W}\left( s,x;\mathbf{z}^{n}\right) _{s,t}$ by
lemma $\left( \ref{Lemma_Joint_RP_convergence}\right) $ and Lyons' limit
theorem and the RHS to $\pi _{W}\left( s,x;\mathbf{z}\right) _{s,t}$; we can
conclude $\mathbf{Q}_{s,t}^{n;W}\equiv \mathbf{Q}_{s,t}^{W}$ for $s,t\in
D^{n}$. A similar argument shows $\mathbf{P}_{s,t}^{n;V}\equiv \mathbf{P}%
_{s,t}^{V}$ for $s,t\in D^{n}$. We finish the argument in the same way as in
the previous example: solutions of $\pi _{V,\left( W\right) }\left(
s,x;\left( \xi ^{n},\mathbf{z}^{n}\right) \right) $ converge uniformly to $%
\pi _{V,\left( W\right) }\left( s,x;\left( \xi ,\mathbf{z}\right) \right) $.
On neighbouring points $s,t\in D^{n},$ $\pi _{V,\left( W\right) }\left(
s,x;\left( \xi ^{n},\mathbf{z}^{n}\right) \right) _{s,t}$ can be idenitified
as 
\begin{equation*}
\left[ \mathbf{Q}_{s,t}^{n,W}\circ \mathbf{P}_{s,t}^{n,V}\right] \left(
x\right) =\left[ \mathbf{Q}_{s,t}^{W}\circ \mathbf{P}_{s,t}^{V}\right]
\left( x\right) .
\end{equation*}%
Hence, for every $t\in \left[ 0,T\right] $%
\begin{equation*}
y_{t}^{n\text{;Split}}:=\dprod\limits_{k=0}^{\left\lfloor t/n\right\rfloor
-1}\left[ \mathbf{Q}_{k/n,\left( k+1\right) /n}^{W}\circ \mathbf{P}%
_{k/n,\left( k+1\right) /n}^{V}\right] \left( y_{0}\right) \rightarrow \pi
_{V,\left( W\right) }\left( 0,y_{0};\left( \xi ,\mathbf{z}\right) \right)
_{t}\text{ as }n\rightarrow \infty \text{ a.s.}
\end{equation*}%
Moreover, $\sup_{n}\left\vert y^{n\text{;Split}}\right\vert _{p-\text{var;}%
\left[ 0,T\right] }<\infty $ which implies by interpolation convergence in $%
\left( p+\epsilon \right) $-variation norm of $y^{n\text{;Split}}$ for every 
$\varepsilon >0$.

\begin{remark}
Similarly, one shows convergence of a splitting scheme, running the
semigroups in the different order $\mathbf{P}^{V}\circ \mathbf{Q}^{W}$. Note
also, that we restrict ourselves in this article to Lie splitting schemes
but the methods can be easily modified to include Strang splitting (see \cite%
{McLachlanQuispel02:SplittingMethods} for the difference between Lie- and
Strang splitting schemes) by using an appropiate modification of the time
change. Further, we just deal with equidistant partitions. Numeruous
variations of all this are possible (as long as one can show convergence in
a rough path topology of the approximating sequence) and such modifications
are of great importance for rates of convergence - a topic which we hope to
address in future work.
\end{remark}

\section{Rough partial differential equations}

One could hope for similar results for SPDEs in Stratonovich form, 
\begin{eqnarray}
du &=&F\left( t,x,u,Du,D^{2}u\right) dt+\sum_{k=1}^{d}\Lambda _{k}\left(
t,x,u,Du\right) \circ dB^{k}  \label{EqRPDEXi} \\
u\left( 0,x\right) &=&u_{0}\left( x\right) \text{ on }\mathbb{R}^{e}\text{ \ 
}  \notag
\end{eqnarray}%
Indeed, splitting for equation $\left( \ref{EqRPDEXi}\right) $ has been
treated in \cite{KrylovGyongy03:OnSplittingUp} for the case of $F$ being a
linear operator. To give meaning to $\left( \ref{EqRPDEXi}\right) $ for
nonlinear $F$ one can introduce the concept of (rough) viscosity solutions
(cf.\ \cite{LionsSoug98:ViscSPDEfirstOne,
LionsSoug98:ViscSPDEsecondOne,LionsSoug98:ViscSPDEsecondOne,LionsSoug00:ViscSPDESeminlinearNoise,LionsSoug00:ViscSPDEUniqueness}
and \cite{FrizOberhauserCaruan09:ViscSPDE} or \cite%
{FrizOberhauser10:StabilityFiltering}). Let us informally discuss the idea
before we give the precise definition in section \ref{SectionMainResult}: a
real-valued, bounded and continuous function $u$ on $\left[ 0,T\right]
\times \mathbb{R}^{e}$ is called a solution if it is the uniform limit
(locally on compacts) of (standard) viscosity solutions $\left( u^{n}\right) 
$ of the equations 
\begin{eqnarray*}
du^{n} &=&F\left( t,x,u^{n},Du^{n},D^{2}u^{n}\right) d\xi
^{n}+\sum_{i=1}^{d}\Lambda _{k}\left( t,x,u,Du\right) dz^{n;i}, \\
u^{n}\left( 0,x\right) &=&u_{0}\left( x\right) \text{ on }\left( 0,T\right]
\times \mathbb{R}^{e},
\end{eqnarray*}%
where $\left( z^{n}\right) \subset C^{\infty }\left( \left[ 0,T\right] ,%
\mathbb{R}^{d}\right) $ and $\left( \xi ^{n}\right) \subset C^{\infty
}\left( \left[ 0,T\right] ,\mathbb{R}\right) $ are sequences of smooth
driving signals, converging to a weak geometric rough path $\left( \xi ,%
\mathbf{z}\right) $ (see $\left( \ref{EqCondYoungPair}\right) $). Formally
we write 
\begin{equation*}
du=F\left( t,x,u,Du,D^{2}u\right) d\xi +\Lambda \left( t,x,u,Du\right) d%
\mathbf{z.}
\end{equation*}%
(a solutions is then also a solution in the sense of \cite%
{LionsSoug98:ViscSPDEsecondOne} when $\mathbf{z}$ is an enhanced Brownian
motion). In the case when the rough path $\mathbf{z}$ is an enhanced
Brownian motion and $\xi \left( t\right) =t$, this gives a natural concept
of Stratonovich solutions to $\left( \ref{EqRPDEXi}\right) $. However, to
apply the methods outlined in the sections above to derive a splitting
method, care has to be taken: firstly, in the RDE case \textquotedblleft any
sequence" $\left( \xi ^{n},z^{n}\right) $ of smooth paths converging to $%
\left( \xi ,\mathbf{z}\right) $ gives rise to a solution, but in the RPDE
case, if $\dot{\xi}_{t}^{n}<0$, one can not expect to treat even the simple
second order equation%
\begin{equation}
du^{n}=D^{2}u^{n}d\xi ^{n}  \label{EqSmoothPDE}
\end{equation}%
since it is no longer degenerate elliptic ($\dot{\xi}_{t}^{n}<0$ amounts to
running the heat equation backwards in time). Secondly, assume we want to
use Lie-splitting on dyadic partitions, i.e.\ approximate $\xi \left(
t\right) =t$ with $\xi ^{n}\left( t\right) =a\left( n^{-1},t\right) $. This
introduces a discontinuous time-dependence ($\dot{\xi}^{n}$ does not exists
on points of the partition)\footnote{%
One could avoid discontinuous time-dependence by restricting the class of
splitting schemes (i.e. the class with $\left( \dot{\xi}^{n}\right) \subset
C^{1}$). However, nearly all popular schemes (Strang, Lie, etc.) would then
not be covered.} in equation $\left( \ref{EqSmoothPDE}\right) .$ Such
time-discontinuities are in general difficult to handle in a viscosity
setting. Thirdly, one has to show continuous dependence of the solution%
\footnote{%
The results in \cite{FrizOberhauserCaruan09:ViscSPDE} and \cite%
{FrizOberhauser10:StabilityFiltering} do not cover this due the
time-discontinuity of the approximating sequence $\left( d\xi ^{n}\right) $.}
of $\left( \ref{EqRPDEXi}\right) ,$ not only on $\mathbf{z}$ but continuous
dependence on $\left( \xi ,\mathbf{z}\right) $ in a rough path sense.

The first point is dealt with by characterizing the class of legit
approximations $\xi ^{n},$ leading to the path space $C_{0}^{1-\text{var;}%
+}\left( \left[ 0,T\right] ,\mathbb{R}\right) $, described in section \ref%
{section_C1-var}. Section \ref{SecPDEdiscontTime} deals with nonlinear PDEs
with a discontinuity-in-time introduced by $\xi ^{n}$ and section \ref%
{SectionMainResult} gives the precise definitions of rough viscosity
solutions and stability, contains the main theorem and examples of stable
RPDEs.

\subsection{The space $C_{0}^{1-\text{var,}+}\left( \left[ 0,T\right] ,%
\mathbb{R}\right) $\label{section_C1-var}}

As pointed out above, we have to avoid to fall outside the scope of
(degenerate) elliptic PDEs. Using the notation $C^{0,1-\text{var}}\left( %
\left[ 0,T\right] ,\mathbb{R}\right) $ for the closure of the space of
smooth paths in variation norm ($\overline{C^{\infty }}^{\left\vert
.\right\vert _{1-var}}\left( \left[ 0,T\right] ,\mathbb{R}\right) $) we
recall that 
\begin{eqnarray*}
W_{0}^{1,1}\left( \left[ 0,T\right] ,\mathbb{R}\right) &\equiv &\left\{ x:%
\left[ 0,T\right] \rightarrow \mathbb{R}\text{, }\exists y\in L^{1}\left( %
\left[ 0,T\right] ,\mathbb{R}\right) \text{ s.t.\ }x\left( t\right)
=\int_{0}^{t}y\left( u\right) du\right\} \\
&=&\left\{ x:\left[ 0,T\right] \rightarrow \mathbb{R}\text{, }x\text{
absolutely continuous, }x\left( 0\right) =0\right\} \\
&=&\overline{\left\{ x:\left[ 0,T\right] \rightarrow \mathbb{R}\text{, }x\in
C^{\infty },\text{ }x\left( 0\right) =0\right\} }^{\left\vert .\right\vert
_{1-\text{var}}} \\
&\equiv &C_{0}^{0,1-var}\left( \left[ 0,T\right] ,\mathbb{R}\right)
\varsubsetneq C_{0}^{1\text{-var}}\left( \left[ 0,T\right] ,\mathbb{R}%
\right) .
\end{eqnarray*}

\begin{definition}
$C_{0}^{1,+}\left( \left[ 0,T\right] ,\mathbb{R}\right) =\left\{ \xi \in
C_{0}^{1}\left( \left[ 0,T\right] ,\mathbb{R}\right) :\xi _{T}=T,\dot{\xi}%
>0\right\} $.
\end{definition}

Note that the paths $a\left( \Delta ,\cdot \right) $ and $b\left( \Delta
,\cdot \right) $ are not elements of $C_{0}^{1,+}\left( \left[ 0,T\right] ,%
\mathbb{R}\right) ,$ but elements of its closure,$\overline{\text{ }%
C_{0}^{1,+}}^{\left\vert .\right\vert _{\infty }}\left( \left[ 0,T\right] ,%
\mathbb{R}\right) ,$ in sup-norm. Working with $C_{0}^{1,+}$ enables us in
section below to give a short proof of existence, uniqueness and stability
of a solution to PDEs of the type $\partial _{t}u=F\dot{\xi}_{t}$ for paths $%
\xi \in \overline{\text{ }C_{0}^{1,+}}^{\left\vert .\right\vert _{\infty
}}\left( \left[ 0,T\right] ,\mathbb{R}\right) $.

\begin{proposition}
\label{PropApproxTime}Denote $C_{0}^{1-\text{var,}+}\left( \left[ 0,T\right]
,\mathbb{R}\right) =\overline{C_{0}^{1,+}}^{\left\vert .\right\vert _{\infty
}}\left( \left[ 0,T\right] ,\mathbb{R}\right) $. Then%
\begin{eqnarray*}
C_{0}^{1-\text{var,}+}\left( \left[ 0,T\right] ,\mathbb{R}\right) &=&\left\{
\xi _{t}\in C_{0}\left( \left[ 0,T\right] ,\mathbb{R}\right) :\xi _{T}=T%
\text{ and }\exists \dot{\xi}\in L^{1}\left( \left[ 0,T\right] ,\mathbb{R}%
_{\geq 0}\right) \text{,}\right. \\
&&\left. \exists a\in C_{0}^{1\text{-var}}\left( \left[ 0,T\right] ,\mathbb{R%
}_{\geq 0}\right) ,\text{ }a\text{ increasing, }\dot{a}=0\text{ a.s. and }%
\xi _{t}=a_{t}+\int_{0}^{t}\dot{\xi}_{u}du\right\}
\end{eqnarray*}%
and $C_{0}^{1-\text{var,}+}\left( \left[ 0,T\right] ,\mathbb{R}\right)
\varsubsetneq C_{0}^{1-\text{var}}\left( \left[ 0,T\right] ,\mathbb{R}%
\right) $.
\end{proposition}

\begin{proof}
$\subset :$ Let $\left( \xi ^{\varepsilon }\right) $ be a Cauchy sequence
wrt.\ $\left\vert .\right\vert _{\infty }$. {}Since $\left( C_{0}\left( %
\left[ 0,T\right] ,\mathbb{R}\right) ,\left\vert .\right\vert _{\infty
}\right) $ is complete, $\xi ^{\varepsilon }$ converges uniformly to some $%
\xi \in C_{0}\left( \left[ 0,T\right] ,\mathbb{R}\right) $. This $\xi $ is
monotone (not necessarily strict) increasing and hence $\left\vert \xi
\right\vert _{1-\text{var};\left[ 0,T\right] }<\infty $ (recall that $\xi
_{T}=T$). Every function of finite $1$-variation is Lebesgue-a.e.
differentiable and has a representation of the form 
\begin{equation*}
\xi _{t}=a_{t}+\int_{\left[ 0,t\right] }\dot{\xi}_{u}du
\end{equation*}%
where $a$ is a function of $1$-variation with $\dot{a}=0$ Lebesgue-a.e. Now $%
\xi _{s,s+h}\geq 0,$ for every $h>0;$ $s\in \left[ 0,1\right) $. Hence we
have $a_{s,s+h}\geq -\int_{s}^{s+h}\dot{\xi}_{u}du$ and sending $%
h\rightarrow 0$ shows together with $\dot{a}=0$ a.s.\ that $a$ is monotone
increasing and this implies $\dot{\xi}_{u}\geq 0$ Lebesgue-a.e.

$\supset :F\left( t\right) :=\xi _{t}$ defines a continuous distribution
function on $\left[ 0,T\right] $ and let $X$ be a random variable with
distribution $F.$ For $\varepsilon >0$ denote by $F^{\varepsilon }$ the
distribution function of the random variable $X+\varepsilon N$ where $N$ is
a standard normal, independent of $X.$ Clearly, $X+\varepsilon N\rightarrow
X $ a.s. as $\varepsilon \rightarrow 0$ and so the $F^{\varepsilon }$%
converge pointwise. By the lemma below this implies uniform convergence of $%
F^{\varepsilon }$ to $F$. It remains to show that $\xi _{t}^{\varepsilon
}:=F^{\varepsilon }\left( t\right) $ is $C^{1}$ but this follows from%
\begin{equation*}
F^{\varepsilon }\left( t\right) =\int_{0}^{t}F\left( t-u\right)
dF_{\varepsilon N}\left( u\right)
\end{equation*}%
where $F_{\varepsilon N}$ is the distribution function of $\varepsilon N$.
\end{proof}

\begin{lemma}
Let $\left( f^{\eta }\right) _{\eta >0}\subset $ $C_{0}\left( \left[ 0,T%
\right] ,\mathbb{R}\right) $ ,$f^{\eta }\left( 1\right) =1$,each $f^{\eta }$
increasing (not necessarily strictly) and assume $f^{\eta }\rightarrow f\in
C_{0}\left( \left[ 0,T\right] ,\mathbb{R}\right) $ pointwise as $\eta
\rightarrow 0$. Then, $\left\vert f^{\eta }-f\right\vert _{\infty ;\left[ 0,T%
\right] }\rightarrow 0$ as $\eta \rightarrow 0.$
\end{lemma}

\begin{proof}
Given $\varepsilon >0$ we can choose a $n\in \mathbb{N}$ big enough s.t. 
\begin{equation*}
\left\vert f\left( \frac{i}{n}\right) -f\left( \frac{i-1}{n}\right)
\right\vert <\frac{\varepsilon }{2}
\end{equation*}%
for every $i\in \left\{ 0,1,\ldots ,n\right\} .$ Now choose $\eta $ small
enough such that 
\begin{equation*}
\left\vert f^{\eta }\left( \frac{i}{n}\right) -f\left( \frac{i}{n}\right)
\right\vert <\frac{\varepsilon }{2}
\end{equation*}%
for all $i\in \left\{ 0,1,\ldots ,n\right\} .$ This implies $\left\vert
f^{\eta }\left( x\right) -f\left( x\right) \right\vert <\varepsilon $ since
every $x$ is an element of (at least one) interval $\left[ \frac{i-1}{n},%
\frac{i}{n}\right] $ and by monotonicity and using above estimates 
\begin{equation*}
f^{\eta }\left( x\right) \leq f^{\eta }\left( \frac{i}{n}\right) \leq
f\left( \frac{i-1}{n}\right) +\frac{\varepsilon }{2}+\frac{\varepsilon }{2}%
\leq f\left( x\right) +\varepsilon \text{.}
\end{equation*}%
Similarly 
\begin{equation*}
f^{\eta }\left( x\right) >f\left( x\right) +\varepsilon
\end{equation*}%
and so $\left\vert f^{\eta }\left( x\right) -f\left( x\right) \right\vert
<\varepsilon $ for all $x\in \left[ 0,T\right] $.
\end{proof}

\begin{remark}
Concerning the choice of notation $C_{0}^{1-\text{var},+},$ note that the
space of paths of finite $1$-variation $C_{0}^{1-var}$ is given as the
closure of $C^{1}$-paths with uniformly bounded $1$-variation. Since paths
in $C_{0}^{1,+}$ have $1$-variation bounded by $T,$ the notation $C_{0}^{1-%
\text{var},+}\left( \left[ 0,T\right] ,\mathbb{R}\right) $ seems natural.
\end{remark}

\begin{remark}
The paths $a\left( \Delta ,\cdot \right) $ and $b\left( \Delta ,\cdot
\right) $ converge in $\left( 1+\varepsilon \right) $-variation to the path $%
id_{1}:t\longmapsto t$ for every $\varepsilon >0$ and therefore also
uniformly (but not in $1$-variation!).
\end{remark}

\begin{remark}
$C_{0}^{1-\text{var},+}\left( \left[ 0,T\right] ,\mathbb{R}\right) $ is not
a linear space but a convex subset of $C_{0}^{1-\text{var}}\left( \left[ 0,T%
\right] ,\mathbb{R}\right) $.
\end{remark}

\begin{remark}
Despite the restriction of $C_{0}^{1-\text{var},+}\left( \left[ 0,T\right] ,%
\mathbb{R}\right) $ to paths with $\xi \left( T\right) =T$ which is
convenient in the proofs, one can handle PDEs with general increasing
processes by rescaling; e.g.\ replace $\xi $ by $\tilde{\xi}\left( t\right)
:=\xi \left( t\right) \frac{T}{\xi \left( T\right) }\in C_{0}^{1-\text{var}%
,+}\left( \left[ 0,T\right] ,\mathbb{R}\right) $ and write $du=F\left(
t,x,u,Du,D^{2}u\right) d\xi =\tilde{F}\left( t,x,u,Du,D^{2}u\right) d\tilde{%
\xi}$ with $\tilde{F}:=F\frac{\xi \left( T\right) }{T}$.
\end{remark}

\subsection{PDEs with discontinuous time-dependence\label{SecPDEdiscontTime}}

This section extends the notion of viscosity solutions to equations of the
form%
\begin{equation*}
du=F\left( t,x,u,Du,D^{2}u\right) d\xi \left( t\right) ,\text{ }u\left(
0,x\right) =u_{0}\left( x\right) ,
\end{equation*}%
with $F$ a continuous function and $\xi \in C_{0}^{1-\text{var,}+}\left( %
\left[ 0,T\right] ,\mathbb{R}\right) $. In the appendix we show that this
solution concept coincides with the notion of generalized viscosity
solutions (going back to \cite{Ishii85:HJBDiscontHamiltonian}) whenever the
latter exists. In view of applications in sections section \ref%
{SectionMainResult} and \ref{Section_SPDE_Apps} and to keep technicalities
down, we focus on time-independent $F$. A proof for time-dependent $F$ is
given in the appendix.

\begin{proposition}
\label{PropL1viscosity}Let $\left( \xi ^{\varepsilon }\right) _{\varepsilon
}\subset C_{0}^{1,+}\left( \left[ 0,T\right] ,\mathbb{R}\right) $ converge
uniformly to some $\xi \in C_{0}^{1-\text{var;}+}\left( \left[ 0,T\right] ,%
\mathbb{R}\right) $ as $\varepsilon \rightarrow 0$. Assume $\left(
v^{\varepsilon }\right) _{\varepsilon }\subset BUC\left( \left[ 0,T\right]
\times \mathbb{R}^{n},\mathbb{R}\right) $ are locally uniformly bounded
viscosity solutions of%
\begin{equation}
\partial _{t}v^{\varepsilon }=F^{\varepsilon }\left( x,v^{\varepsilon
},Dv^{\varepsilon },D^{2}v^{\varepsilon }\right) \dot{\xi}_{t}^{\varepsilon
},\text{ }v^{\varepsilon }\left( 0,x\right) =v_{0}\left( x\right) .
\label{EqPDEZetaEpsilon}
\end{equation}%
with $F^{\varepsilon }:\mathbb{R}^{e}\times \mathbb{R}\times \mathbb{R}%
^{e}\times \mathbb{S}^{e}\mathbb{\rightarrow R}$, $\mathbb{S}^{e}$ denoting
the space of symmetric $\left( e\times e\right) $-matrices, a continuous and
degenerate elliptic function. Further, assume that $F^{\varepsilon }$
converges locally uniformly to a continuous, degenerate elliptic function $F$
and that a comparison result holds for $\partial _{t}-F^{\varepsilon }=0$
and $\partial _{t}-F=0$. Then there exists a$\ v$ such that%
\begin{equation*}
v^{\varepsilon }\rightarrow v\text{ locally uniformly as }\varepsilon
\rightarrow 0.
\end{equation*}%
Further, $v$ does not depend on the choice of the sequence approximating $%
\xi $ and we also write $v\equiv v^{\xi }$ to emphasize the dependence on $%
\xi $ and say that $v$ solves 
\begin{equation*}
dv=F\left( x,v,Dv,D^{2}v\right) d\xi _{t},\text{ }v\left( 0,x\right)
=v_{0}\left( x\right) .
\end{equation*}
\end{proposition}

Prepare the proof with

\begin{lemma}
\label{LemTimeChange}\bigskip Let $\xi \in C_{0}^{1,+}\left( \left[ 0,T%
\right] ,\mathbb{R}\right) $, 
\begin{equation*}
F:\left[ 0,T\right] \times \mathbb{R}^{n}\times \mathbb{R\times \mathbb{R}}%
^{n}\times \mathbb{S}^{n}\rightarrow \mathbb{R}
\end{equation*}%
and let $\tilde{F}\left( t,x,r,p,X\right) =F\left( \xi ^{-1}\left( t\right)
,r,x,p,X\right) $ for $\left( t,x,r,p,X\right) \in \left[ 0,T\right] \times 
\mathbb{R}^{n}\times \mathbb{R\times \mathbb{R}}^{n}\times \mathbb{S}^{n}.$
Then

\begin{enumerate}
\item if $u\in BUC\left( \left[ 0,T\right] \times \mathbb{R}^{n}\right) $ is
a sub- (resp. super) solution of $\partial _{t}-F\dot{\xi}=0,$ $u\left(
0,.\right) =u_{0}\left( .\right) $ then $w\left( t,x\right) :=u\left( \xi
_{t}^{-1},x\right) $ is a sub- (resp.\ super) solution of $\partial _{t}-%
\tilde{F}=0,$ $w\left( 0,.\right) =u_{0}\left( .\right) $.

\item if $w\in BUC\left( \left[ 0,T\right] \times \mathbb{R}^{n}\right) $ is
a sub- (resp. super) solution of $\partial _{t}-\tilde{F}=0$, $w\left(
0,.\right) =w_{0}\left( .\right) $ then $u\left( t,x\right) :=w\left( \xi
_{t},x\right) $ is a sub- (resp. super) solution of $\partial _{t}-F\dot{\xi}%
,$ $u\left( 0,.\right) =w_{0}\left( .\right) $.
\end{enumerate}
\end{lemma}

\begin{proof}
1. Let $\varphi \in C^{1,2}\left( \left[ 0,T\right] \times \mathbb{R}%
^{n}\right) $ and assume $w\left( t,x\right) -\varphi \left( t,x\right) $
attains a local maximum at $\left( \hat{t},\hat{x}\right) \in \left[ 0,T%
\right] \times \mathbb{R}^{n}$. Then 
\begin{equation*}
w\left( \hat{t},\hat{x}\right) -\varphi \left( \hat{t},\hat{x}\right)
=u\left( \xi ^{-1}\left( \hat{t}\right) ,\hat{x}\right) -\varphi \left( \hat{%
t},\hat{x}\right) =u\left( \xi ^{-1}\left( \hat{t}\right) ,\hat{x}\right) -%
\tilde{\varphi}\left( \xi ^{-1}\left( \hat{t}\right) ,\hat{x}\right)
\end{equation*}%
where $\tilde{\varphi}\left( \hat{t},\hat{x}\right) :=\varphi \left( \xi _{%
\hat{t}},\hat{x}\right) .$ Using that $u$ is a subsolution gives%
\begin{eqnarray*}
\partial _{t}\hat{\varphi}\lvert _{\xi _{\hat{t}}^{-1},\hat{x}} &\leq
&F\left( \hat{x},\xi ^{-1}\left( \hat{t}\right) ,u\left( \xi ^{-1}\left( 
\hat{t}\right) ,\hat{x}\right) ,D\tilde{\varphi}\lvert _{\xi _{\hat{t}}^{-1},%
\hat{x}},D^{2}\tilde{\varphi}\lvert _{\xi _{\hat{t}}^{-1},\hat{x}}\right) 
\dot{\xi}_{\xi _{\hat{t}}^{-1}} \\
&=&F\left( \hat{x},\xi ^{-1}\left( \hat{t}\right) ,w\left( \hat{t},\hat{x}%
\right) ,D\varphi \lvert _{\hat{t},\hat{x}},D^{2}\varphi \lvert _{\hat{t},%
\hat{x}}\right) \dot{\xi}_{\xi _{\hat{t}}^{-1}} \\
&=&\tilde{F}\left( \hat{x},\hat{t},w\left( \hat{t},\hat{x}\right) ,D\varphi
\lvert _{\hat{t},\hat{x}},D^{2}\varphi \lvert _{\hat{t},\hat{x}}\right) \dot{%
\xi}_{\xi _{\hat{t}}^{-1}}
\end{eqnarray*}%
where we used that $D\tilde{\varphi}\lvert _{\xi _{\hat{t}}^{-1},\hat{x}%
}=D\varphi \lvert _{\hat{t},\hat{x}}$ and $D^{2}\tilde{\varphi}\lvert _{\xi
_{\hat{t}}^{-1},\hat{x}}=D^{2}\varphi \lvert _{\hat{t},\hat{x}}$. Since $%
\partial _{t}\hat{\varphi}\lvert _{\xi _{t}^{-1},x}=\partial _{t}\varphi
\lvert _{t,x}$ $\dot{\xi}\lvert _{\xi _{t}^{-1}}$ \ and $\dot{\xi}>0$ it
follows that 
\begin{equation*}
\partial _{t}\varphi \lvert _{\hat{t},x}\leq \tilde{F}\left( \hat{x},\hat{t}%
,w\left( \hat{t},\hat{x}\right) ,D\varphi \lvert _{\hat{t},\hat{x}%
},D^{2}\varphi \lvert _{\hat{t},\hat{x}}\right) \text{.}
\end{equation*}%
Now the same argument as above when $u$ is a supersolution.

2. Let $\varphi \in C^{1,2}\left( \left[ 0,T\right] \times \mathbb{R}%
^{n}\right) $ and assume $u\left( t,x\right) -\varphi \left( t,x\right) $
attains a local maximum at $\left( \hat{t},\hat{x}\right) \in \left[ 0,T%
\right] \times \mathbb{R}^{n}$. Then 
\begin{equation*}
u\left( \hat{t},\hat{x}\right) -\varphi \left( \hat{t},\hat{x}\right)
=w\left( \xi \left( \hat{t}\right) ,\hat{x}\right) -\varphi \left( \hat{t},%
\hat{x}\right) =w\left( \xi \left( \hat{t}\right) ,\hat{x}\right) -\tilde{%
\varphi}\left( \xi \left( \hat{t}\right) ,\hat{x}\right)
\end{equation*}%
where $\tilde{\varphi}\left( t,x\right) :=\varphi \left( \xi ^{-1}\left(
t\right) ,x\right) .$ Using that $w$ is a subsolution gives%
\begin{eqnarray*}
\partial _{t}\hat{\varphi}\lvert _{\xi _{\hat{t}},\hat{x}} &\leq &\tilde{F}%
\left( \hat{x},\xi \left( \hat{t}\right) ,w\left( \xi \left( \hat{t}\right) ,%
\hat{x}\right) ,D\tilde{\varphi}\lvert _{\xi _{\hat{t}},\hat{x}},D^{2}\tilde{%
\varphi}\lvert _{\xi \left( \hat{t}\right) ,\hat{x}}\right) \\
&=&F\left( \hat{x},\hat{t},u\left( \hat{t},\hat{x}\right) ,D\varphi \lvert _{%
\hat{t},\hat{x}},D^{2}\varphi \lvert _{\hat{t},\hat{x}}\right)
\end{eqnarray*}%
where we used that $D\tilde{\varphi}\lvert _{\xi _{\hat{t}}^{-1},\hat{x}%
}=D\varphi \lvert _{\hat{t},\hat{x}}$ and $D^{2}\tilde{\varphi}\lvert _{\xi
_{\hat{t}}^{-1},\hat{x}}=D^{2}\varphi \lvert _{\hat{t},\hat{x}}$. Since $%
\dot{\xi}>0$ and $\partial _{t}\hat{\varphi}\lvert _{\xi _{t},x}=\partial
_{t}\varphi \lvert _{t,x}$ $\left( \xi ^{-1}\right) ^{\prime }\lvert _{\xi
_{t}}=$ \ $\partial _{t}\varphi \lvert _{t,x}\left( \dot{\xi}\left( t\right)
\right) ^{-1}$ it follows that 
\begin{equation*}
\partial _{t}\varphi \lvert _{\hat{t},x}\leq F\left( \hat{x},\hat{t},u\left( 
\hat{t},\hat{x}\right) ,D\varphi \lvert _{\hat{t},\hat{x}},D^{2}\varphi
\lvert _{\hat{t},\hat{x}}\right) \dot{\xi}\left( \hat{t}\right) \text{.}
\end{equation*}%
Now the same argument as above when $w$ is a supersolution.
\end{proof}

\begin{proof}[Proof of Proposition \protect\ref{PropL1viscosity}]
Set $w^{\varepsilon }\left( t,x\right) :=v^{\varepsilon }\left( \left( \xi
^{\varepsilon }\right) ^{-1}\left( t\right) ,x\right) $, by lemma \ref%
{LemTimeChange}, 
\begin{equation*}
v^{\varepsilon }\text{ is a solution of }\partial _{t}-F^{\varepsilon }\dot{%
\xi}^{\varepsilon }=0\text{ iff }w^{\varepsilon }\text{ is a solution of }%
\partial _{t}-F^{\varepsilon }=0\text{.}
\end{equation*}%
Let 
\begin{equation*}
\overline{w}:=\lim \sup_{\varepsilon }\text{ }^{\ast }w^{\varepsilon }\text{
and }\underline{w}:=\lim \inf_{\varepsilon }\text{ }_{\ast }w^{\varepsilon }
\end{equation*}%
and note that 
\begin{equation*}
F^{\varepsilon }\left( x,r,p,X\right) \rightarrow F\left( x,r,p,X\right) 
\text{ locally uniformly.}
\end{equation*}%
Standard viscosity theory tells us that $\overline{w}$ and $\underline{w}$
are sub- resp. supersolutions of $\partial _{t}-F=0$. Using the method of
semi-relaxed limits (by definition, $\overline{w}\geq \underline{w}$ and the
reversed inequality follows from comparison which holds by assumption)
conclude that $w\left( t,x\right) :=\overline{w}\left( t,x\right) =%
\underline{w}\left( t,x\right) .$ Further, using a Dini-type argument, for
every compact set $K\subset \mathbb{R}^{n}$, 
\begin{equation*}
\left\vert w^{\varepsilon }-w\right\vert _{\infty ;\left[ 0,T\right] \times
K}=\sup_{t\in \left[ 0,T\right] ,\text{ }x\in K}\left\vert w^{\varepsilon
}\left( t,x\right) -w\left( t,x\right) \right\vert \rightarrow 0\text{ as }%
\varepsilon \rightarrow 0\text{.}
\end{equation*}%
Now define 
\begin{equation*}
v\left( t,x\right) :=w\left( \xi \left( t\right) ,x\right) .
\end{equation*}%
and we get the claimed convergence $v^{\varepsilon }\rightarrow v$.
\end{proof}

\section{Splitting RPDEs\label{SectionMainResult}}

\subsection{Rough viscosity solutions and stability}

Solutions of RDEs can be defined as limit points of ODEs. Similarly one can
define solutions of rough partial differential equations as limit points of
PDE solutions.

\begin{definition}
\label{Def_Stable}Let $\mathbf{z\in }C_{0}^{0,p\text{-var}}\left( \left[ 0,T%
\right] ,G^{\left[ p\right] }\left( \mathbb{R}^{d}\right) \right) ,$ $\xi
\in C_{0}^{1-\text{var;}+}\left( \left[ 0,T\right] ,\mathbb{R}\right) $ and
denote $\left( \xi ,\mathbf{z}\right) \in C_{0}^{0,p\text{-var}}\left( \left[
0,T\right] ,G^{\left[ p\right] }\left( \mathbb{R}^{d+1}\right) \right) $ the
Young pairing (given canoncially via Young integration). Further, let $%
\left( \xi ^{\varepsilon },z^{\varepsilon }\right) \subset $ $%
C_{0}^{1,+}\left( \left[ 0,T\right] ,\mathbb{R}\right) \times C_{0}^{1\text{%
-var}}\left( \left[ 0,T\right] ,\mathbb{R}^{d}\right) ,$ converge to $\left(
\xi ,\mathbf{z}\right) $ in the sense of $\left( \ref%
{Lemma_Joint_RP_convergence}\right) $ and assume the PDE%
\begin{eqnarray*}
\mathrm{d}u^{\varepsilon } &=&F\left( x,u^{\varepsilon },Du^{\varepsilon
},D^{2}u^{\varepsilon }\right) \mathrm{d}\xi ^{\varepsilon
}+\sum_{k=1}^{d}\Lambda _{k}\left( t,x,u^{\varepsilon },Du^{\varepsilon
}\right) \mathrm{d}z^{\varepsilon ;k}\text{ on }\left( 0,T\right] \times 
\mathbb{R}^{n},\text{ } \\
u\left( 0,x\right) &=&u_{0}\left( x\right) \in BUC\left( \mathbb{R}^{e},%
\mathbb{R}\right)
\end{eqnarray*}%
has a unique solution $u^{\varepsilon }\in BUC\left( \left[ 0,T\right]
\times \mathbb{R}^{e};\mathbb{R}\right) $ for every $\varepsilon >0$. We
call every limit point $u$ of $\left( u^{\varepsilon }\right) $ (in $BUC$
topology) a solution of the RPDE 
\begin{equation}
\mathrm{d}u=F\left( x,u,Du,D^{2}u\right) \mathrm{d}\xi +\Lambda \left(
t,x,u,Du\right) \mathrm{d}\mathbf{z}\text{ on }\left( 0,T\right] \times 
\mathbb{R}^{n},\text{ }u\left( 0,x\right) =u_{0}\left( x\right) .
\label{Eq_Stable_RPDE}
\end{equation}%
If additionally, the limit is unique, does not depend on the choice of the
approximating sequence $\left( \xi ^{\varepsilon },z^{\varepsilon }\right) $
and the map%
\begin{equation*}
\left( \xi ,\mathbf{z}\right) \in C_{0}^{1-\text{var,}+}\left( \left[ 0,T%
\right] ,\mathbb{R}\right) \times C_{0}^{0,p\text{-var}}\left( \left[ 0,T%
\right] ,G^{\left[ p\right] }\left( \mathbb{R}^{d}\right) \right) \mapsto
u\in BUC\left( \left[ 0,T\right] \times \mathbb{R}^{e},\mathbb{R}\right)
\end{equation*}%
is continuous then we say that the RPDE $\left( \ref{Eq_Stable_RPDE}\right) $
is stable in a rough path sense and we also write $u=u^{\xi ,\mathbf{z}}$
(or $u=u^{\mathbf{z}}$ when $\xi \left( t\right) =t$) to emphasize
dependence on the rough path $\left( \xi ,\mathbf{z}\right) $.
\end{definition}

\subsection{The main theorem}

We are now able to formulate our main theorem. The proof is an easy
consquence of the results in the previous sections. In section \ref%
{SectionMainResult} we show that the assumptions are satisfied for a large
class of RPDEs.

\begin{theorem}
\label{Thm_Splitting}Let $\mathbf{z\in }C_{0}^{0,p\text{-var}}\left( \left[
0,T\right] ,G^{\left[ p\right] }\left( \mathbb{R}^{d}\right) \right) ,\xi
\in C_{0}^{+,1-\text{var}}\left( \left[ 0,T\right] ,\mathbb{R}\right) $ and
assume $u\in BUC$ is the unique solution of the stable (in the sense of
definition \ref{Def_Stable}) RPDE, 
\begin{equation}
\mathrm{d}u=F\left( x,u,Du,D^{2}u\right) \mathrm{d}\xi +\Lambda \left(
t,x,u,Du\right) \mathrm{d}\mathbf{z}\text{ on }\left( 0,T\right] \times 
\mathbb{R}^{e},u\left( 0,x\right) =u_{0}\left( x\right) \in BUC.
\label{Eq_StableRPDE}
\end{equation}%
Assume further that also the two (R)PDEs given by setting either $F\equiv 0$
or $\Lambda \equiv 0$ in $\left( \ref{Eq_StableRPDE}\right) $ are stable.
Denote $\left\{ \mathbf{P}_{s,t},0\leq u\leq T\right\} $ and $\left\{ 
\mathbf{Q}_{s,t},0\leq s\leq t\leq T\right\} $ the solution operators%
\begin{eqnarray*}
\mathbf{P}_{s,t} &:&BUC\left( \mathbb{R}^{e},\mathbb{R}\right) \rightarrow
BUC\left( \left[ 0,T\right] \times \mathbb{R}^{e},\mathbb{R}\right) ,\text{ }%
\varphi \longmapsto v, \\
\mathbf{Q}_{s,t} &:&BUC\left( \mathbb{R}^{e},\mathbb{R}\right) \rightarrow
BUC\left( \left[ 0,T\right] \times \mathbb{R}^{e},\mathbb{R}\right) ,\text{ }%
\phi \longmapsto w,
\end{eqnarray*}%
with%
\begin{eqnarray*}
&&dv=F\left( x,v,Dv,D^{2}v\right) d\xi ,\text{ }v\left( 0,x\right) =\varphi
\left( x\right) , \\
&&dw=\Lambda \left( t,x,w,Dw\right) d\mathbf{z,}\text{ }w\left( s,x\right)
=\phi \left( x\right) .
\end{eqnarray*}%
and set%
\begin{equation*}
u^{n;\text{Split}}\left( t,x\right) :=\dprod\limits_{i=0}^{\left\lfloor
tn^{-1}\right\rfloor -1}\left[ \mathbf{Q}_{in^{-1},\left( i+1\right)
n^{-1}}\circ \mathbf{P}_{in^{-1},\left( i+1\right) n^{-1}}\right] \left(
u_{0}\left( x\right) \right) .
\end{equation*}%
Then%
\begin{equation*}
u^{n;\text{Split}}\rightarrow u\text{ locally uniformly as }n\rightarrow
\infty \text{.}
\end{equation*}
\end{theorem}

\begin{proof}
Define $\mathbf{z}^{n}=\mathbf{z}\left( b\left( n^{-1},t\right) \right) $, $%
\xi ^{n}\left( t\right) =\xi \left( a\left( n^{-1},t\right) \right) $. By
lemma \ref{Lemma_Joint_RP_convergence}, $\left( \xi ^{n},\mathbf{z}%
^{n}\right) \rightarrow _{n}\left( \xi ,\mathbf{z}\right) $ in the sense of $%
\left( \ref{Lemma_Joint_RP_convergence}\right) $ and by stability, the
solutions $u^{n}$ of 
\begin{equation*}
\mathrm{d}u^{n}=F\left( x,u^{n},Du^{n},D^{2}u^{n}\right) \mathrm{d}\xi
^{n}+\Lambda \left( t,x,u^{n},Du^{n}\right) \mathrm{d}z^{n}\text{ on }\left(
0,T\right] \times \mathbb{R}^{n},u\left( 0,x\right) =u_{0}\left( x\right)
\end{equation*}%
converge to $u,$ the solution of $\left( \ref{Eq_Stable_RPDE}\right) $.\ Now
for each given $n$ one can identify on points of the dissection $\left\{ k%
\frac{T}{n},i=0,\ldots ,n\right\} $ the solutions of $u^{n;\text{Split}}$
with $u^{n}$ and by assumptions $u^{n}$ converges locally uniformly to $u.$
\end{proof}

\subsection{Examples of stable RPDEs}

This section shows stability in a rough path sense for a large class of
RPDEs. Splitting results then follow readily by theorem \ref{Thm_Splitting}.
Throughout this section $\mathbf{z}$ is a geometric $p-$rough path, i.e. $%
\mathbf{z\in }C_{0}^{0,p-\text{var}}\left( \left[ 0,T\right] ,G^{p}\left( 
\mathbb{R}^{d}\right) \right) ,$ $p\geq 1$.

\begin{proposition}
\label{Prop_Linear_is_Stable}Let 
\begin{eqnarray*}
L\left( x,r,p,X\right) &=&\mathrm{Tr}\left[ A\left( x\right) ^{T}X\right]
+b\left( x\right) \cdot p+f\left( x,r\right) \\
\Lambda _{k}\left( t,x,r,p\right) &=&\left( p\cdot \sigma _{k}\left(
t,x\right) \right) +r\,\nu _{k}\left( t,x\right) +g_{k}\left( t,x\right)
\end{eqnarray*}%
with $A\left( x\right) =\bar{\sigma}\left( x\right) \bar{\sigma}^{T}\left(
x\right) \in \mathbb{S}^{e}$ and $\bar{\sigma}:\mathbb{R}^{e}\rightarrow 
\mathbb{R}^{e\times e^{\prime }}$, $b\left( x\right) :\mathbb{R}%
^{e}\rightarrow \mathbb{R}^{e}$ bounded, Lipschitz continuous in $x$. Also
assume that $f:\mathbb{R}^{e}\times \mathbb{R}\rightarrow \mathbb{R}$ is
continuous, bounded whenever $r$\ remains bounded, and with a lower
Lipschitz bound, i.e.\ $\exists C<0$ s.t.%
\begin{equation*}
f\left( x,r\right) -f\left( x,s\right) \geq C\left( r-s\right) \text{ for
all }r\geq s,\text{ }x\in \mathbb{R}^{e}\text{.}
\end{equation*}%
and that the coefficients of $\Lambda =\left( \Lambda _{1},\dots ,\Lambda
_{d}\right) ,$ thats is $\sigma ,\nu $ and $g,$ have $Lip^{\gamma }$%
-regularity for $\gamma >p+2.$ Then the RPDE%
\begin{equation}
du=L\left( x,u,Du,D^{2}u\right) dt+\Lambda \left( t,x,u,Du\right) d\mathbf{%
z,\,\,\,}u\left( 0,\cdot \right) \equiv u_{0}  \label{Eq_LinRPDE}
\end{equation}%
is stable in a rough path sense and has a unique solution $u^{\mathbf{z}}\in
BUC\left( \left[ 0,T\right] \times \mathbb{R}^{e},\mathbb{R}\right) $.
\end{proposition}

\begin{proof}
We use the same technique of \textquotedblleft rough semi-relaxed limits" as
in \cite{FrizOberhauser10:StabilityFiltering}; for details of the transform
of $u^{\varepsilon }$ to $\tilde{v}^{\varepsilon }$ we refer to \cite%
{FrizOberhauser10:StabilityFiltering}: the key remark being that $%
u^{\varepsilon }$ is a solution of%
\begin{equation}
du^{\varepsilon }=L\left( x,u^{\varepsilon },Du^{\varepsilon
},D^{2}u^{\varepsilon }\right) d\xi ^{\varepsilon }+\Lambda \left(
t,x,u^{\varepsilon },Du^{\varepsilon }\right) dz^{\varepsilon }\mathbf{%
,\,\,\,}u\left( 0,\cdot \right) \equiv u_{0}\left( \cdot \right)
\label{Eq_Linear_RPDE_Eps}
\end{equation}%
iff 
\begin{equation*}
\tilde{v}^{\varepsilon }\left( t,x\right) :=\left( \phi ^{\varepsilon
}\right) ^{-1}\left( t,u^{\varepsilon }\left( t,\psi ^{\varepsilon }\left(
t,x\right) \right) ,x\right) +\alpha ^{\varepsilon }\left( t,x\right) ;
\end{equation*}%
is a solution of 
\begin{equation*}
d\tilde{v}^{\varepsilon }=\tilde{L}^{\varepsilon }\left( x,\tilde{v}%
^{\varepsilon },D\tilde{v}^{\varepsilon },D^{2}\tilde{v}^{\varepsilon
}\right) d\xi ^{\varepsilon }\mathbf{,\,\,\,}\tilde{v}^{\varepsilon }\left(
0,\cdot \right) \equiv u_{0}\left( \cdot \right) .
\end{equation*}%
Here $\tilde{L}^{\varepsilon }$ is a linear operator with coefficients
determined by the characteristics of the PDE $\partial w=\Lambda \left(
t,x,w^{\varepsilon },Dw^{\varepsilon }\right) dz^{\varepsilon }$ and $\psi
^{\varepsilon },\phi ^{\varepsilon },\alpha ^{\alpha }$ are ODE flows
converging to RDE flows $\psi ^{\mathbf{z}},\phi ^{\mathbf{z}},\alpha ^{%
\mathbf{z}}$ which depend on $\mathbf{z}$ (but not the approximating
sequence of $\left( z^{\varepsilon }\right) $). Further a comparison
principle applies to $\partial _{t}-\tilde{L}^{\varepsilon }=0$. The
assumptions of proposition \ref{PropL1viscosity} are then fulfilled, $\tilde{%
L}^{\varepsilon }\rightarrow \tilde{L}$ locally uniformly and using the
method of semirelaxed limits,%
\begin{equation*}
\tilde{v}^{\varepsilon }\rightarrow \tilde{v}\text{ locally uniformly.}
\end{equation*}%
Unwrapping the transformation, that is, setting%
\begin{equation}
u^{\mathbf{z}}\left( t,x\right) :=\phi ^{\mathbf{z}}\left( t,\tilde{v}\left(
t,\left( \psi ^{\mathbf{z}}\right) ^{-1}\left( t,x\right) \right) -\alpha ^{%
\mathbf{z}}\left( t,\left( \psi ^{\mathbf{z}}\right) ^{-1}\left( t,x\right)
\right) \right) ,  \label{Eq_Unwrap}
\end{equation}%
finishes the proof since stability of the RPDE follows directly from the
representation $\left( \ref{Eq_Unwrap}\right) $.
\end{proof}

\begin{proposition}
\label{Prop_HJB}Let 
\begin{equation*}
F\left( x,r,p,X\right) =\inf_{\alpha \in \mathcal{A}}\left\{ \mathrm{Tr}%
\left[ A\left( x;\alpha \right) ^{T}X\right] +b\left( x;\alpha \right) \cdot
p+f\left( x,r;\alpha \right) \right\}
\end{equation*}%
where $A,\sigma $ and $b,f$ satisfy the assumption of proposition \ref%
{Prop_Linear_is_Stable} uniformly with respect to $\alpha \in \mathcal{A}$
and $\nu =\left( \nu _{k}\right) _{k=1}^{d}\subset Lip^{\gamma }\left( 
\mathbb{R}^{n},\mathbb{R}^{n}\right) $ $\gamma >p+2$. Then the RPDE%
\begin{equation*}
du=F\left( x,u,Du,D^{2}u\right) dt+Du\cdot \nu \left( x\right) d\mathbf{%
z,\,\,\,}u\left( 0,\cdot \right) \equiv u_{0}
\end{equation*}%
has a unique solution $u^{\mathbf{z}}\in BUC\left( \left[ 0,T\right] \times 
\mathbb{R}^{n}\right) $ and is stable in a rough path sense.
\end{proposition}

\begin{proof}
Similar to the proof above.
\end{proof}

\section{Applications to stochastic PDEs\label{Section_SPDE_Apps}}

The typical applications to SPDEs are path-by-path, i.e.\ by taking $\mathbf{%
z}$ to be a realization of a continuous semi-martingale $Y$ and its
stochastic area, say $\mathbf{Y}\left( \omega \right) =\left( Y,A\right) $;
the most prominent example being Brownian motion and L\'{e}vy's area. Taking
the linear case as an example, the stability result of proposition \ref%
{Prop_Linear_is_Stable} allows to identify 
\begin{equation*}
du=L\left( t,x,u,Du,D^{2}u\right) dt+\Lambda \left( t,x,u,Du\right) d\mathbf{%
z,\,\,\,}u\left( 0,\cdot \right) \equiv u_{0}
\end{equation*}%
with $\mathbf{z}=\mathbf{Y}\left( \omega \right) $ as \textit{Stratonovich
solution} to the SPDE%
\begin{equation*}
du=L\left( t,x,u,Du,D^{2}u\right) dt+\Lambda \left( t,x,u,Du\right) \circ
dY,\,\,\,u\left( 0,\cdot \right) =u_{0}.
\end{equation*}%
Indeed, under the stated assumptions, the Wong-Zakai approximations, in
which $Y$ is replaced by its piecewise linear approximation, based on some
mesh $\left\{ 0,\frac{T}{n},\frac{2T}{n}\dots ,T\right\} $, the approximate
solution will converge (locally uniformly on $\left[ 0,T\right] \times 
\mathbb{R}^{n}$ and in probability, say) to the solution of%
\begin{equation*}
du=L\left( t,x,u,Du,D^{2}u\right) dt+\Lambda \left( t,x,u,Du\right) d\mathbf{%
Y},\,\,\,u\left( 0,\cdot \right) =u_{0},
\end{equation*}%
as constructed in proposition \ref{Prop_Linear_is_Stable}. In view of
well-known Wong-Zakai approximation results for SPDEs, ranging from \cite%
{MR1313027, MR1353194} to \cite{MR2052265, MR2268661}, the rough PDE
solution is then identified as Stratonovich solution. (At least for $L$
uniformly elliptic: the (Stratonovich) integral interpretations can break
down in degenerate situations; as example, consider non-differentiable
initial data $u_{0}$ and the (one-dimensional) random transport equation $%
du=u_{x}\circ dB$ with explicit "Stratonovich" solution $u_{0}\left(
x+B_{t}\right) $. A similar situation occurs for the classical transport
equation $\dot{u}=u_{x}$, of course.) Motivated by this, if $u^{\mathbf{z}}$
is a RPDE solution of

\begin{equation*}
du=F\left( t,x,u,Du,D^{2}u\right) dt+\Lambda \left( t,x,u,Du\right) d\mathbf{%
z},\,\,\,u\left( 0,\cdot \right) =u_{0},
\end{equation*}%
then we call $u^{\mathbf{z}}$ with $\mathbf{z}=\mathbf{Y}\left( \omega
\right) $ as \textit{Stratonovich solution} and write 
\begin{equation*}
du=F\left( t,x,u,Du,D^{2}u\right) dt+\Lambda \left( t,x,u,Du\right) \circ
dY,\,\,\,u\left( 0,\cdot \right) =u_{0}.
\end{equation*}

The following example was suggested in \cite{MR1659958} and carefully worked
out in \cite{MR1920103, MR2285722}.

\begin{example}[Pathwise stochastic control]
\label{ExStochControl}Consider%
\begin{equation*}
dX=b\left( X;\alpha \right) dt+W\left( X;\alpha \right) \circ d\tilde{B}%
+V\left( X\right) \circ dB,
\end{equation*}%
where $b,W,V$ are (collections of) sufficiently nice vector fields (with $%
b,W $ dependent on a suitable control $\alpha =\alpha \left( t\right) \in 
\mathcal{A}$, applied at time $t$) and $\tilde{B},B$ are multi-dimensional
(independent)\ Brownian motions. Define\footnote{%
Remark that any optimal control $\alpha \left( \cdot \right) $ here will
depend on knowledge of the entire path of $B$. Such anticipative control
problems and their link to classical stochastic control problems were
discussed early on by Davis and Burnstein \cite{MR1275134}.}%
\begin{equation*}
v\left( x,t;B\right) =\inf_{\alpha \in \mathcal{A}}\mathbb{E}\left[ \left.
\left( g\left( X_{T}^{x,t}\right) +\int_{t}^{T}f\left( X_{s}^{x,t},\alpha
_{s}\right) ds\right) \right\vert B\right]
\end{equation*}%
where $X^{x,t}$ denotes the solution process to the above SDE started at $%
X\left( t\right) =x$. Then, at least by a formal computation,%
\begin{equation*}
dv+\inf_{\alpha \in \mathcal{A}}\left[ b\left( x,\alpha \right)
Dv+L_{a}v+f\left( x,\alpha \right) \right] dt+Dv\cdot \nu \left( x\right)
\circ dB=0
\end{equation*}%
with terminal data $v\left( \cdot ,T\right) \equiv g$, and $L_{\alpha }=\sum
W_{i}^{2}$ in H\"{o}rmander form. Setting $u\left( x,t\right) =v\left(
x,T-t\right) $ turns this into the initial value (Cauchy) problem, 
\begin{equation*}
du=\inf_{\alpha \in \mathcal{A}}\left[ b\left( x,\alpha \right)
Du+L_{a}u+f\left( x,\alpha \right) \right] dt+Du\cdot V\left( x\right) \circ
dB_{T-\cdot }
\end{equation*}%
with initial data $\,u\left( \cdot ,0\right) \equiv g$; and hence of a form
which is covered by theorem \ref{Prop_HJB}.(Moreover, the rough driving
signal in proposition \ref{Prop_HJB} is taken as $\mathbf{z}_{t}:=\mathbf{B}%
_{T-t}\left( \omega \right) $ where $\mathbf{B}\left( \omega \right) $ is a
fixed Brownian motion).
\end{example}

Using theorem \ref{Thm_Splitting} we immediately get a splitting result:

\begin{example}[Splitting HJB-equations]
Let $B$ be standard $d-$dimensional Brownian motion. Then the SPDE%
\begin{eqnarray}
du &=&\inf_{\alpha \in \mathcal{A}}\left\{ \mathrm{Tr}\left[ \sigma \left(
x;\alpha \right) \sigma \left( x;\alpha \right) ^{T}D^{2}u\right] +b\left(
x;\alpha \right) \cdot Du+f\left( x;\alpha \right) \right\} dt+\left(
Du\cdot \nu \left( x\right) \right) \circ dB,  \notag \\
\text{ }u\left( \omega ;0,x\right) &=&u_{0}\left( x\right) ,
\label{EqHJBRPDE}
\end{eqnarray}%
has a unique a unique solution $u$ if $\sigma \left( x,\alpha \right) :%
\mathbb{R}^{e}\times \mathcal{A}\rightarrow \mathbb{R}^{e\times e^{\prime }}$
and $b\left( x,\alpha \right) :\mathbb{R}^{e}\times \mathcal{A}\rightarrow 
\mathbb{R}^{e}$ are Lipschitz continuous in $x$, uniformly in $\alpha \in 
\mathcal{A},$ $\nu =\left( \nu _{1},\dots ,\nu _{d}\right) \subset \mathrm{%
Lip}^{\gamma }\left( \mathbb{R}^{e};\mathbb{R}^{e}\right) $ with $\gamma >4$%
. Denote $\left\{ \mathbf{P}_{u},u\geq 0\right\} $ the solution operator%
\footnote{%
Note that in the case $\xi =t$ one can use a one-parameter semigroup since $%
\dot{\xi}^{n}$ $=2$ on the time interval on which the approximation evolve.}
of 
\begin{equation}
du=\inf_{\gamma \in \mathcal{A}}\left\{ \mathrm{Tr}\left[ \sigma \left(
x;\alpha \right) \sigma \left( x;\alpha \right) ^{T}D^{2}u\right] +b\left(
x;a\right) \cdot Du+f\left( x;\alpha \right) \right\} dt,  \label{EqHJB}
\end{equation}%
i.e.\ $\mathbf{P}_{t}u_{0}\left( .\right) =u\left( t,.\right) ,$ and $%
\left\{ \mathbf{Q}_{s,t},0\leq s\leq t\leq T\right\} $ the solution operator
given by $\mathbf{Q}_{s,t}\varphi \left( .\right) =\varphi \left( \pi
_{-V}\left( s,.;B\right) _{t}\right) ,$ with $\pi _{-V}\left( s,x;B\right)
_{t}$ the SDE solution of 
\begin{equation}
dy=-V\left( y\right) \circ dB\text{, }y_{s}=x\in \mathbb{R}^{e}\text{.}
\label{EqRDe}
\end{equation}%
Then\footnote{%
Apriori the leftmost two terms would have to be $\mathbf{Q}_{\left\lfloor
t/n\right\rfloor ,t}\circ \mathbf{P}_{t-\left\lfloor t/n\right\rfloor n}$.
However, the claimed convergence follows from Lyons' limit theorem.},%
\begin{equation*}
u^{n;\text{Split}}\left( t,x\right) :=\dprod\limits_{i=0}^{\left\lfloor
t/n\right\rfloor -1}\left[ \mathbf{Q}_{i/n,i/n+1/n}\circ \mathbf{P}_{1/n}%
\right] \left( u_{0}\left( x\right) \right) \rightarrow u\left( t,x\right) 
\text{ as }n\rightarrow \infty
\end{equation*}%
and the convergence also holds locally uniformly.
\end{example}

Thus, equation $\left( \ref{EqHJBRPDE}\right) $ can be approximated by
solutions of a standard HJB equation $\left( \ref{EqHJBRPDE}\right) $ and by
solutions of the RDE $\left( \ref{EqRDe}\right) $ (for numerical schemes for
HJB, see \cite{FlemingSoner06:ControlBook}).

\begin{example}[Linear SPDEs, Filtering]
Let $L$ and $\Lambda $ be as in proposition \ref{Prop_Linear_is_Stable}.
Then there exists a unique solution to 
\begin{equation*}
du=L\left( x,u,Du,D^{2}u\right) dt+\sum_{k=1}^{d}\Lambda _{k}\left(
t,x,u,Du\right) \circ dB^{k},u\left( 0,.\right) =u_{0}\left( .\right)
\end{equation*}%
Denote by $\left\{ \mathbf{P}_{u},0\leq u\leq T\right\} $ the solution
operator%
\begin{equation*}
\varphi \longmapsto v\text{ with }v\text{ solution of }\partial v=L\left(
x,v,Dv,D^{2}v\right) dt,v\left( 0,\cdot \right) =\varphi \left( \cdot \right)
\end{equation*}%
and by $\left\{ \mathbf{Q}_{s,t},0\leq s\leq t\right\} $ the solution
operator 
\begin{equation*}
\varphi \longmapsto y\text{ with }y\text{ solution of }dy=\Lambda \left(
t,x,u,Du\right) \circ dB,v\left( 0,\cdot \right) =\varphi \left( \cdot
\right)
\end{equation*}%
Then,%
\begin{equation*}
u^{n;\text{Split}}\left( t,x\right) :=\dprod\limits_{i=0}^{\left\lfloor
t/n\right\rfloor -1}\left[ \mathbf{Q}_{i/n,i/n+1/n}\circ \mathbf{P}_{1/n}%
\right] \left( u_{0}\left( x\right) \right) \rightarrow u\left( t,x\right) 
\text{ as }n\rightarrow \infty
\end{equation*}%
and the convergence also holds locally uniformly.
\end{example}

\appendix

\section{Appendix: Time-dependent\label{app_timedep} $F$}

To deal with time-dependent $F$ we need the additional assumption of uniform
bounds on the derivatives of the approximating sequence $\xi ^{\varepsilon
}. $

\begin{proposition}
\label{Prop_TimedepF}Let $\left( \xi ^{\varepsilon }\right) _{\varepsilon
}\subset C_{0}^{1,+}\left( \left[ 0,T\right] ,\mathbb{R}\right) ,$ $%
\sup_{\varepsilon }\left\vert \dot{\xi}^{\varepsilon }\right\vert _{\infty ;%
\left[ 0,T\right] }<\infty ,$ converge uniformly to some $\xi \in C_{0}^{1-%
\text{var;}+}\left( \left[ 0,T\right] ,\mathbb{R}\right) $ as $\varepsilon
\rightarrow 0.$ Assume $\left( v^{\varepsilon }\right) _{\varepsilon
}\subset BUC\left( \left[ 0,T\right] \times \mathbb{R}^{e},\mathbb{R}\right) 
$ are locally uniformly bounded viscosity solutions of%
\begin{equation*}
\partial _{t}v^{\varepsilon }=F^{\varepsilon }\left( t,x,v^{\varepsilon
},Dv^{\varepsilon },D^{2}v^{\varepsilon }\right) \dot{\xi}_{t}^{\varepsilon
},\text{ }v^{\varepsilon }\left( 0,x\right) =v_{0}\left( x\right) .
\end{equation*}%
with $F^{\varepsilon }:\left[ 0,T\right] \times \mathbb{R}^{e}\times \mathbb{%
R}\times \mathbb{R}^{e}\times \mathbb{S}^{e}\mathbb{\rightarrow R}$ a
continuous and degenerate elliptic function. Further, assume that $%
F^{\varepsilon }$ converges locally uniformly to a continuous, degenerate
elliptic function $F$ and that a comparison result holds for $\partial
_{t}-F^{\varepsilon }=0$ and $\partial _{t}-F=0$. Then there exists a$\ v$
such that%
\begin{equation*}
v^{\varepsilon }\rightarrow v\text{ locally uniformly as }\varepsilon
\rightarrow 0.
\end{equation*}%
Further, $v$ does not depend on the choice of the sequence approximating $%
\xi $ and we also write $v\equiv v^{\xi }$ to emphasize the dependence on $%
\xi $ and say that $v$ solves 
\begin{equation*}
dv=F\left( t,x,v,Dv,D^{2}v\right) d\xi _{t},\text{ }v\left( 0,x\right)
=v_{0}\left( x\right) .
\end{equation*}
\end{proposition}

\begin{proof}
Set $w^{\varepsilon }\left( t,x\right) :=v^{\varepsilon }\left( \left( \xi
^{\varepsilon }\right) ^{-1}\left( t\right) ,x\right) $ and divide $\left[
0,T\right] $ into intervals on which $\xi $ is strictly increasing resp.\
constant, i.e.\ $0=s_{1}\leq t_{1}\leq s_{2}\leq \cdots \leq t_{n}=T,$ $\xi $
strictly increasing on $\left[ s_{i},t_{i}\right] $, constant on $\left[
t_{i},s_{i+1}\right] $. By lemma \ref{LemTimeChange} on intervals $\left[
s_{i},t_{i}\right] $ 
\begin{equation*}
v^{\varepsilon }\text{ is a solution of }\partial _{t}-F\dot{\xi}%
^{\varepsilon }=0\text{ iff }w^{\varepsilon }\text{ is a solution of }%
\partial _{t}-\tilde{F}^{\varepsilon }=0
\end{equation*}%
where $\tilde{F}^{\varepsilon }\left( t,x,r,p,X\right) =F\left( \left( \xi
^{\varepsilon }\right) ^{-1}\left( t\right) ,x,r,p,X\right) $. Let 
\begin{equation*}
\overline{w}:=\lim \sup_{\varepsilon }\text{ }^{\ast }w^{\varepsilon }\text{
and }\underline{w}:=\lim \inf_{\varepsilon }\text{ }_{\ast }w^{\varepsilon }
\end{equation*}%
and note that on intervals $\left[ s_{i},t_{i}\right] $ 
\begin{equation*}
\tilde{F}^{\varepsilon }\left( t,x,r,p,X\right) \rightarrow F\left( \xi
^{-1}\left( t\right) ,x,r,p,X\right) =:\tilde{F}\left( t,x,r,p,X\right) 
\text{ locally uniformly.}
\end{equation*}%
Standard viscosity theory tells us that on intervals $\left[ s_{i},t_{i}%
\right] ,$ $\overline{w}$ and $\underline{w}$ are sub- resp. supersolutions
of $\partial _{t}-\tilde{F}=0$. Using the method of semi-relaxed limits (by
definition, $\overline{w}\geq \underline{w}$ and the reversed inequality
follows from comparison) conclude that $w:=\overline{w}=\underline{w}$ and
that for every compact set $K\subset \mathbb{R}^{n}$ (by using a Dini-type
argument), 
\begin{equation*}
\left\vert w^{\varepsilon }-w\right\vert _{\infty ;\left[ s_{i},t_{i}\right]
\times K}=\sup_{t\in \left[ s_{i},t_{i}\right] ,\text{ }x\in K}\left\vert
w^{\varepsilon }\left( t,x\right) -w\left( t,x\right) \right\vert
\rightarrow 0\text{ as }\varepsilon \rightarrow 0.
\end{equation*}%
Now define 
\begin{equation*}
v\left( t,x\right) :=\left\{ 
\begin{array}{ll}
w\left( \xi \left( t\right) ,x\right) , & \xi ^{-1}\left( s_{i}\right) \leq
t\leq \xi ^{-1}\left( t_{i}\right) \\ 
w\left( \xi \left( t_{i}\right) ,x\right) , & \xi ^{-1}\left( t_{i}\right)
<t<\xi ^{-1}\left( s_{i+1}\right)%
\end{array}%
\right.
\end{equation*}%
We get the claimed convergence $v^{\varepsilon }\rightarrow v$ on intervalls 
$\left[ \xi ^{-1}\left( s_{i}\right) ,\xi ^{-1}\left( t_{i}\right) \right] $%
. However, on $\left[ \xi ^{-1}\left( t_{i}\right) ,\xi ^{-1}\left(
s_{i+1}\right) \right] ,$ $v^{\varepsilon }$ is by definition viscosity
solution of $\partial _{t}-F\dot{\xi}^{\varepsilon }=0$ with initial
condition $v^{\varepsilon }\left( \xi ^{-1}\left( t_{i}\right) ,.\right) $
and $F\dot{\xi}^{\varepsilon }\rightarrow _{\varepsilon }0$ locally
uniformly since $\sup_{t}\left\vert \dot{\xi}^{\varepsilon }\left( t\right)
\right\vert $ is uniformly bounded in $\varepsilon $ by assumption. Hence
the standard stability result of viscosity theory applies and $%
v^{\varepsilon }$ converges locally uniformly on $\left[ \xi ^{-1}\left(
t_{i}\right) ,\xi ^{-1}\left( s_{i+1}\right) \right] $ against the
constant-in-time function $v^{\varepsilon }\left( \xi ^{-1}\left(
t_{i}\right) ,.\right) ,$ the only solution to $\partial _{t}=0$ with
initial condition $v^{\varepsilon }\left( \xi ^{-1}\left( t_{i}\right)
,.\right) $. This proves the claimed convergence. Further, note that $v$ is
given as the unique viscosity solution of $\partial _{t}-\tilde{F}=0$, hence
every other sequence approximating $\xi $ will lead to the same limit.
\end{proof}

The proof of the main theorem (theorem \ref{Thm_Splitting}) and applications
to examples adapt now in a straightforward way to time-dependent $F$.

\section{Appendix: Generalized viscosity solutions}

Section \ref{SecPDEdiscontTime} and appendix \ref{app_timedep} extend the
notion of viscosity solutions to equations of the form%
\begin{equation}
du=F\left( t,x,u,Du,D^{2}u\right) d\xi \left( t\right) ,\text{ }u\left(
0,x\right) =u_{0}\left( x\right) .  \label{EqPDE_Badtime}
\end{equation}%
with $\xi \in C^{1-v\text{ar};+}\left( \left[ 0,T\right] ,\mathbb{R}\right)
. $ Generalizations of viscosity solutions go back to \cite%
{Ishii85:HJBDiscontHamiltonian},\cite{LionPerthame87:RemarksonHJB} and for
the parabolic case \cite%
{Nunziante92:ExistenceUniquenessParabolicViscosityDiscTimedependence}. Let
us recall the definition given in \cite%
{Nunziante92:ExistenceUniquenessParabolicViscosityDiscTimedependence}.

\begin{condition}
\label{Cond_1}$F\left( \cdot ,x,r,p,X\right) \in L^{1}\left( \left(
0,T\right) ,\mathbb{R}\right) $ for all $\left( x,r,p,X\right) \in \mathbb{R}%
^{e}\times \mathbb{R\times \mathbb{R}}^{e}\times \mathbb{S}^{e}$ and $F$ is
continuous on $\mathbb{R}^{e}\times \mathbb{R\times \mathbb{R}}^{e}\times 
\mathbb{S}^{e}$ for almost all $t\in \left( 0,T\right) $
\end{condition}

\begin{condition}
\label{Cond_2}$F\left( t,x,\cdot ,p,X\right) $ is nondecreasing on $\mathbb{R%
}$ for all $t\in \left( 0,T\right) $ and for all $\left( x,p,X\right) \in 
\mathbb{R}^{e}\mathbb{\times \mathbb{R}}^{e}\times \mathbb{S}^{e}$.
\end{condition}

\begin{definition}
\label{Def_Generalized_Visc}Let $F$ satisfy conditions \ref{Cond_1} and \ref%
{Cond_2}. A locally bounded uniformly upper semicontinuous function $u\in
BUC\left( \left[ 0,T\right] \times \mathbb{R}^{e}\right) $ is called a
generalized subsolution of%
\begin{equation}
du=F\left( t,x,u,Du,D^{2}u\right) ,\text{ }u\left( 0,x\right) =u_{0}\left(
x\right)  \label{Eq_BadPDE}
\end{equation}%
if for any $\left( \hat{t},\hat{x}\right) \in \left[ 0,T\right] \times 
\mathbb{R}^{e},b\in L^{1}\left( \left[ 0,T\right] ,\mathbb{R}\right) ,\phi
\in C^{2}\left( \mathbb{R}^{e},\mathbb{R}\right) ,G:\left[ 0,T\right] \times 
\mathbb{R}^{e}\times \mathbb{R\times \mathbb{R}}^{e}\times \mathbb{S}%
^{e}\rightarrow \mathbb{R}$ continuous and degenerate parabolic, such that%
\begin{equation*}
u\left( t,x\right) +\int_{0}^{t}b\left( r\right) dr-\phi \left( t,x\right) 
\text{ attains a local maximum at }\left( \hat{t},\hat{x}\right) ,
\end{equation*}%
and%
\begin{eqnarray*}
b\left( t\right) +G\left( t,x,r,p,X\right) &\leq &F\left( t,x,r,p,X\right) 
\text{ for a.e. }t\in B_{\delta }\left( \hat{t}\right) \text{ and } \\
\text{for all }\left( x,r,p,X\right) &\in &B_{\delta }\left( \hat{x},u\left( 
\hat{t},\hat{x}\right) ,D\phi \lvert _{\hat{t},\hat{x}},D^{2}\phi \lvert _{%
\hat{t},\hat{x}}\right) \text{ for some }\delta >0
\end{eqnarray*}%
it follows that 
\begin{equation*}
b\left( \hat{t}\right) +G\left( \hat{x},u\left( \hat{t},\hat{x}\right)
,D\phi \lvert _{\hat{t},\hat{x}},D^{2}\phi \lvert _{\hat{t},\hat{x}}\right)
\leq 0.
\end{equation*}%
A locally bounded uniformly lower semicontinuous function is called a
supersolution if the above estimates hold when one replaces maximum by
minimum and reverses the inequality sign.
\end{definition}

Note that equation $\left( \ref{Eq_BadPDE}\right) $ is covered by this
definition. However, it is quite cumbersome to derive existence, comparison
and stability results in this very general setting and in the case of
interest to us, the time-discontinuouity only appears multiplicatively.

\begin{proposition}
Under the assumptions of proposition \ref{Prop_TimedepF} and additionaly $%
\xi \in W^{1,1}$ the function $u=u^{\xi }$ is a viscosity solution of 
\begin{equation*}
du=F\left( t,x,u,Du,D^{2}u\right) d\xi _{t},\text{ }u\left( 0,x\right)
=u_{0}\left( x\right)
\end{equation*}%
in the sense of definition \ref{Def_Generalized_Visc}.
\end{proposition}

\begin{proof}
We partition $\left[ 0,T\right] $ into $0\leq s_{1}\leq t_{1}\leq \cdots
\leq s_{n}\leq t_{n}\leq T$ such that $\xi $ is increasing on $\left[
s_{i},t_{i}\right] $, constant on $\left[ t_{i},s_{i+1}\right] $. Say $%
u\left( t,x\right) +\int_{0}^{t}b\left( r\right) dr-\phi \left( t,x\right) $
attains a local maximum at $\left( \hat{t},\hat{x}\right) $. If $\hat{t}\in %
\left[ s_{i},t_{i}\right] $ by construction $u\left( t,x\right) \equiv
w\left( t,\xi _{t}\right) $ with $w$ a viscosity subsolution of $\partial
_{t}-\tilde{F}=0$, $\tilde{F}\left( t,r,x,p,X\right) =F\left( \xi
^{-1}\left( t\right) ,r,x,p,X\right) $, hence also a generalized subsolution
and using that $\xi $ is invertible on $\left[ s_{i},t_{i}\right] $ one sees
by a change of variable that also $u$ is a generalized subsolution on $\left[
s_{i},t_{i}\right] $ of $\partial _{t}-\tilde{F}=0.$

If $\hat{t}\in \left[ t_{i},s_{i+1}\right] ,$ then $\xi $ is constant, hence 
$\dot{\xi}=0$ a.s.\ and so $F\left( t,x,r,p,X\right) \dot{\xi}\left(
t\right) =0$ for a.e.\ $t\in B_{\delta }\left( \hat{t}\right) $ and $u$ is a
generalized subsolution on that interval. This shows that $u$ is a
generalized subsolution and the same argument shows that $u$ is a
generalized supersolution.
\end{proof}

\bibliographystyle{aalpha}
\bibliography{acompat,roughpaths,rpath}

\end{document}